\documentclass[11pt]{amsart}
\usepackage[utf8]{inputenc}
\pdfoutput = 1

\title{Stability of Canonical Forms}

\usepackage{amssymb, amstext, amscd, amsmath, amssymb, amsaddr}
\usepackage{mathtools, color, paralist, dsfont, rotating, mathrsfs}
\usepackage{verbatim}
\usepackage{enumerate}
\usepackage{relsize}
\usepackage{setspace}
\usepackage{graphicx}
\usepackage{float}
\usepackage{fancyhdr}
\usepackage{caption}
\usepackage{ragged2e}
\usepackage{subfigure}

\usepackage{bm}

\usepackage[a4paper,top=3cm,bottom=3cm,left=2.5cm,right=2.5cm,marginparwidth=1.75cm]{geometry}
\usepackage[hidelinks]{hyperref}

\usepackage{cite}
\theoremstyle{plain}
\newtheorem{theorem}{Theorem}[section]
\newtheorem{corollary}[theorem]{Corollary}
\newtheorem{proposition}[theorem]{Proposition}
\newtheorem{lemma}[theorem]{Lemma}
\theoremstyle{definition}
\newtheorem{definition}[theorem]{Definition}
\newtheorem{example}[theorem]{Example}

\newtheorem{remark}[theorem]{Remark}

\theoremstyle{remark}

\mathtoolsset{centercolon}

\newcommand{\A}{{\mathcal{A}}}
\newcommand{\B}{{\mathcal{B}}}
\newcommand{\C}{{\mathcal{C}}}

\newcommand{\E}{{\mathcal{E}}}

\renewcommand{\O}{{\mathcal{O}}}

\renewcommand{\S}{{\mathcal{S}}}
\newcommand{\T}{{\mathcal{T}}}

\newcommand{\X}{{\mathcal{X}}}

\newcommand{\bR}{\mathbb{R}}

\newcommand{\bE}{\mathbb{E}}

\newcommand{\Ba}{{\mathbf{a}}}

\newcommand{\Bu}{{\mathbf{u}}}

\newcommand{\sE}{{\mathscr{E}}}

\newcommand{\sT}{\mathscr{T}}

\newcommand{\AND}{\text{ and }}

\newcommand{\FORAL}{\text{ for all }}

\newcommand{\OR}{\text{ or }}

\newcommand{\bnorm}[2]{\big\|#2\big\|_{#1}}
\newcommand{\Bnorm}[2]{\Big\|#2\Big\|_{#1}}
\newcommand{\norm}[2]{\|#2\|_{#1}}

\renewcommand{\epsilon}{\varepsilon}
\newcommand{\thmref}[2]{#1 \ref{#2}}
\renewcommand{\vec}{\textbf{vec}}
\newcommand{\pd}[2]{\frac{\partial #1}{\partial #2}}
\renewcommand{\tt}[1]{\tsr{T}^{(#1)}}

\usepackage[normalem]{ulem}
\usepackage[dvipsnames]{xcolor}

\newcommand{\vcr}[1]{\bm{#1}}
\newcommand{\mat}[1]{\bm{#1}}

\newcommand{\tsr}[1]{\bm{\mathcal{#1}}}
\newcommand{\set}[1]{#1}
\newcommand{\tn}[1]{\mathcal{#1}}

\begin{document}

\title{On Stability of Tensor Networks and Canonical Forms}
\author[Y. Zhang and E. Solomonik]{Yifan Zhang, Edgar Solomonik}
\address{
Department of Computer Science\\
University of Illinois at Urbana-Champaign}
\email{yifan8@illinois.edu, solomon2@illinois.edu}
\maketitle

\onehalfspacing

\begin{abstract}
    Tensor networks such as matrix product states (MPS) and projected entangled pair states (PEPS) are commonly used to approximate quantum systems.
    These networks are optimized in methods such as DMRG or evolved by local operators.
    We provide bounds on the conditioning of tensor network representations to sitewise perturbations.
    These bounds characterize the extent to which local approximation error in the tensor sites of a tensor network can be amplified to error in the tensor it represents.
    In known tensor network methods, canonical forms of tensor network are used to minimize such error amplification.
    However, canonical forms are difficult to obtain for many tensor networks of interest.
    We quantify the extent to which error can be amplified in general tensor networks, yielding estimates of the benefit of the use of canonical forms.
    For the MPS and PEPS tensor networks, we provide simple forms on the worst-case error amplification.
    Beyond theoretical error bounds, we experimentally study the dependence of the error on the size of the network for perturbed random MPS tensor networks.
\end{abstract}


\section{Introduction} \label{s.intro}




Tensor networks are widely utilized in computational physics to approximate quantum states and represent Hamiltonian operators~\cite{markov2008simulating,orus2014practical,HUCKLE2013750}.
1D and 2D tensor networks are most prevalent and are referred to as matrix product states (MPS), also known as tensor train (TT)~\cite{holtz2012alternating}, and projected entangled pair states (PEPS), respectively~\cite{verstraete2008matrix,lubasch2014algorithms,verstraete2004renormalization}.
However, other types of tensor networks are also used to represent different families of functions~\cite{ye2018tensor}.

Multiple tensor networks may represent the same tensor, even when the structure of the network is fixed. 
A simple example is the matrix-matrix product, or equivalently a 2-site MPS, $\mat{T} = \mat{A}\mat{B}$. 
One can take any invertible matrix $\mat M$ of compatible shape, and set $\mat{A}' = \mat{A}\mat M$ and $\mat{B}' = \mat M^{-1}\mat{B}$, then we obtain another representation $\mat{T} = \mat{A}'\mat{B}'$. 
This degree of freedom is called the gauge freedom of the network, and tensor network algorithms attempt to restrict the gauge to tensor networks that do not suffer much from numerical instabilities, see for example~\cite{evenbly2018gauge}, 
as we will see in our discussion that different gauges may have different properties regarding numerical stability.
A canonical form is a particular gauge with a prescribed center (a tensor node or a set of tensor nodes), such that all but the center nodes in the network contract to an isometric operator (isometric environment matrix). Rigorous definition is given in definition \ref{def.canon}.
The choice of gauge can involve trade-offs between computational cost and the achieved accuracy and stability.

In computational physics and chemistry, one is often interested in the minimum eigenpair (ground state and its energy) or near-minimum eigenpairs (excited states and their energy).
The density matrix normalization group (DMRG) algorithm computes these quantities by optimizing 1D tensor network representations of these states~\cite{RevModPhys.77.259,white1992density,white1993density}.
To optimize each site, DMRG puts the tensor network into a canonical form with that site as the center, computing an orthogonal projection of the eigenvalue problem to a reduced eigenproblem. 
When leveraging tensor networks including loops, such as PEPS, canonicalization of the tensor network is costly and difficult to perform accurately (for MPS, a sequence of singular value decompositions suffices)~\cite{haghshenas2019conversion,zaletel2019isometric,hyatt2019dmrg}.
Without a canonical form, a nonorthogonal projection can be used to yield a reduced generalized eigenproblem for each site~\cite{DAVIDSON198949}.
The second benefit of canonicalization is an improvement in numerical stability~\cite{haghshenas2019conversion,RevModPhys.77.259}. 
Further, in methods for solving linear systems and least squares problems with MPS representations, canonical forms provide guarantees that reduced problems are at least as well conditioned as the overall matrix problem~\cite{holtz2012alternating}.
We aim to further quantify benefits of canonical forms and conditioning of general tensor networks.

We study the stability of tensor networks and attainable accuracy in tensor network eigenvalue problems, with a special focus on the effect of canonical forms.
First, in Section~\ref{s.wrst}, we consider the conditioning of general tensor networks and prove worst-case error bounds on the amplification of a perturbation of a tensor network state with respect to a perturbation on one or all of its sites, 
and we demonstrate that our worst-case bounds are tight. 
Our results show that this amplification depends on the norms of environments: matrices describing the relation between a site and the overall tensor network state.
Next in Section \ref{s.ave}, we utilize tools from Section \ref{s.wrst} to provide an average case analysis that predicts typical error amplification.

In Section~\ref{s.stab_met} and \ref{s.appl}, we apply the general bounds to tensor network computations with MPS and PEPS.
The bounds allow us to ascertain the effect of truncation of a single site or of all sites onto the overall tensor state.
Further, we derive worst-case bounds on the attainable accuracy of eigenpair computations using DMRG for MPS in both general and canonical forms. Such results can be directly generalized to columnwise canonicalized PEPS introduced in~\cite{haghshenas2019conversion,zaletel2019isometric,hyatt2019dmrg}.
Our analysis demonstrates that canonical forms generally provide benefits in both attainable accuracy in DMRG and time evolution algorithms, and numerical stability against sitewise perturbation.
Specifically, for single site perturbation, canonical forms centered at the perturbed site guarantees to reduce the worst-case error by a factor of
$\frac{\norm{2}{\mat M}\norm{F}{\tsr T}}{\norm{F}{\mat M\cdot \tsr T}}$, 
where $\tsr T$ us the site we are perturbing, and $\mat M$ is the environment of the site. While for all-site perturbation, canonical forms improve the worst-case error by a factor of $\frac{\norm{2}{\mat M}\norm{F}{\tsr T}}{\sqrt{D}\norm{F}{\mat M\cdot \tsr T}}$, where $D$ is the dimension of contracted mode 
between $\mat{M}$ and $\tsr T$. Though technically this factor can be smaller than 1, canonical form still brings benefits when the few largest singular values of $\mat{M}$ are dominating (i.e. $\norm{2}{\mat{M}} \approx \norm{F}{\mat{M}}$), say when they are exponentially decreasing. This is justified in our numerical examples in Section \ref{s.num}. 
Further, the attainable accuracy for canonical form is better by a similar factor as in the single-site perturbation.

Our theoretical analysis is confirmed by numerical experiments in Section~\ref{s.num}.
We first 
compute the worst-case errors in the overall state due to single-site perturbations
for randomly generated 1D tensor network states (MPS) and their canonical forms. This models the worst-case error of center truncation.
We quantify the stability benefits of canonical forms and the sensitivity of this benefit to increase in bond dimension or number of sites.
The benefit is observed to increase relative to bond dimension, but not strongly related to site count.
For an MPS with 32 $\sim$ 80 sites, bond dimension $2^7 = 128$, truncation at the central node yields a factor of $4 \sim 4.5$ less worst-case error if the MPS is canonicalized (towards the central node), averaged over all sampled MPS networks.
Second, we illustrate the errors for all-site perturbation by perturbing randomly sampled MPS networks with randomly generated noise. 
We observe that despite in theory canonical form may not have a better stability for some particular MPS, in most of the cases we can expect better stability when in a canonical form. In addition, such advantage in stability tends to increase as bond dimension grows.
Lastly, we numerically confirm our result in Section \ref{s.ave} on average-case error.
Overall, we confirm the stability benefits of canonicalization and add to the theoretical understanding of its dependence on the particular tensor network.


\section{Tensor Networks} \label{s.prelim}
In denoting vectors, matrices, and tensors, we follow notational conventions common in tensor decomposition literature~\cite{doi:10.1137/07070111X}, but employ different elementwise indexing notation.
We use calligraphic letters $(\tn T,~\tn C,~\ldots)$ to denote tensor networks formed by many tensor nodes. 
In addition, we use script letters $(\sT,~\sE,~\ldots)$ and also lower-case Latin letters $(f,~g,~\ldots)$ to denote functions.
We say a tensor $\tsr{T}$ has order $k$ if it has $k$ modes/indices/legs, each of which has a corresponding dimension. 
For a tensor network $\tn{T}$, we use its bold letter $\tsr{T}$ to denote the contraction output.

A tensor network (TN) defines the contraction of a set of tensors, in a certain way prescribed by the network, into a single output tensor. 
It corresponds to an undirected graph where vertices (also referred to as nodes or sites) are tensors, while edges can either go between vertices, in which case they denote summation indices (contracted legs), or can be adjacent to a single vertex (uncontracted legs), in which case they correspond to modes/legs of the tensor given by contraction of the network.
\begin{definition}\label{d.tn}
We denote a tensor network (TN) represented by graph $G$ as
\[\tn T = (G,\tt{1},~\tt{2},~\ldots,~\tt{n})\]
and define TN-forming function $\sT_G(\tt{1},~\tt{2},~\ldots,~\tt{n})=\tsr T$ as the map replacing each vertex of $G$ with tensors (of compatible shape) $\tt{1},~\tt{2},~\ldots,~\tt{n}$ and contracting to form the output tensor $\tsr T$.
Let $G = (\set{V}, \set{E})$ be a graph with $\set{V} = \{v_1,~v_2,~\ldots,~v_n\}$.
Let $\set{E}=\set{C}\cup \set{U}$, where $\set{C}$ are edges connecting two distinct vertices, representing contracted legs, and $\set{U}$ are self-loops representing uncontracted legs.
For each $v_j \in \set{V}$, let $\set{C}_j := \{ (u, w) \in \set{C}~|~u = v_j \OR w = v_j \}$ be the set of edges (contracted legs) adjacent to $v_j$.
Suppose we put $\tt{j}$ at vertex $v_j$.
Let $\set{U}_j$ be the set of uncontracted legs (loops) of vertex $\tt{j}$, then $\set{U} = \bigcup_{j = 1}^n \set{U}_j$.
Generically, for a subgraph $H = (\set{V}', \set{E}')$, we also use notation $\set{U}_H = \bigcup_{v_j \in \set{V}'} \set{U}_j$ and $\set{C}_H = \bigcup_{v_j \in \set{V}'} \set{C}_j$. 
We associate a unique index with each uncontracted or contracted index. 
We index an entry of a tensor $\tt{j}$ by $\left(\tt{j}\right)_{\set{C}_j \cup \set{U}_j}$, e.g., if $\set{C}_j$ contains two edges indexed respectively as $p$ and $q$, while $\set{U}_j$ contains uncontracted legs indexed $r$ and $s$, this entry is $\left(\tt{j}\right)_{\{pq\} \cup \{rs\}} \equiv \left(\tt{j}\right)_{pqrs}$.
The function $\sT_G$ then outputs a mode $|\set{U}|$ tensor $\T$,
and each entry can be computed by
\begin{equation} \label{eq.defT1}
    \tsr T_{\set{U}} = \sum_{\set{C}} \prod_{j = 1}^n \left(\tt{j}\right)_{\set{C}_j \cup \set{U}_j},
\end{equation}
where summing over set $\set{C}$ means summing over all indices (i.e. edges) in $\set{C}$, and by our indexing convention, $\tsr T_{\set{U}}$ is an element of $\tsr T$ indexed by elements in $\set{U}$. 
\end{definition}
We note the fact that the right-hand side of \eqref{eq.defT1} is linear in each $\tt{j}$, yielding the following proposition.
\begin{proposition} \label{prop.lin}
The function $\sT_G$ is multilinear.
\end{proposition}

To formally define the notion of environment and canonical form, we need to define a sub-network induced by a subset of vertices. Let $H_G(\set{V}') = (\set{V}', \set{E}')$ extract a subgraph of $G$ induced by taking the vertices $\set{V}'$, their loops, and their edges to other vertices in $\set{V}'$. 
\begin{definition} [Sub-network from induced subgraph] \label{def.subn}
We define sub-network $S[G, {\set V'}]$ with $G=(V,E)$ to be a graph $(\set{V}', \set{E}'')$, where $\set{E}''$ is obtained by adding to subgraph $H_G(\set{V}') = (\set{V}', \set{E}')$ loops $(u, u)$ for each $u\in \set{V}'$ and $(u,v) \in \set{E}$ with $v \notin \set{V}'$.
\end{definition}
%

With these definitions, we are able to define the environment of vertices $V'$ as follows.

\begin{definition} [Environment tensor of a sub-network]
    Given a set of vertices $\set{V}' = \{v_1,~v_2,~\ldots,~v_k\}$, the environment tensor is $\sT_{S[G,\set{V}\setminus\set{V}']}(\tt{k+1},~\tt{k+2},~\ldots,~\tt{n})$.
\end{definition}

\begin{definition} [Environment matrix of a sub-network] \label{def.sub_tns_env_mat}
Let $H=H_G(\set V')$ be a proper induced subgraph of $G$. 
Let $\mat{I}_{H}$ be the identity matrix $\bigotimes_{e \in \set{U}_{H}} \mat{I}_e$. Let $\mat{N}_{H}$ be the matricization of the environment tensor $\sT_{S[G,\set{V}\setminus\set{V}']}(\tt{k+1},~\tt{k+2},~\ldots,~\tt{n})$ where the rows iterate over indices $\set{U}_{G\setminus H}$ ($G\setminus H$ is the induced subgraph of $G$ by vertices $V\setminus V'$). The environment matrix is $\mat{M}_{\set V'}=  \mat{N}_{H} \otimes \mat{I}_H$. 
\end{definition}

With these definition, the tensor represented by the tensor network can be regarded as the result of a matrix-vector product with the environment tensor and the contracted extracted sub-network induced from vertices $\set V'=\{v_1,\cdots, v_k\}$.
\begin{equation} \label{eq.matvec_repr}
    \vec(\sT_{G}(\tt{1},~\tt{2},~\ldots,~\tt{n})) = \mat{M}_{\set V'} \cdot \vec(\sT_{G}(\tt{1},~\tt{2},~\ldots,~\tt{k})).
\end{equation}
For sake of brevity, we assume in the following discussion the underlying graph $G$ is fixed with $n$ vertices, and write $\sT_G$ simply as $\sT$. Finally, we can define the canonical form of a tensor network.

\begin{definition} [Canonical form] \label{def.canon}
    A tensor network is said to be in a canonical form centered at vertex $\tau$ if the environment matrix $\mat{M}_{(\tau)}$ is an isometry (orthogonal matrix with at least as many rows as columns).
\end{definition}

We will limit our discussion to cases where $\tt{j}$ is a single tensor node in $\T$ for a cleaner notation. 
It is evident that what is discussed below will still apply if $\tt{j}$ is replaced with a sub-network.
As an example, one can regard each column of a PEPS as one node, then reduce a PEPS to a MPS.

\begin{definition} [Environment matrix] \label{def.env_mat}
    Let $\T = (G,~\tt{1},~\tt{2},~\ldots,~\tt{n})$ be a general tensor network. The environment matrix of the entire network, or simply the environment matrix, is defined as
    \begin{equation}
        \mat{M}_\T := \begin{bmatrix} \mat{M}_{\tt{1}} ~ \mat{M}_{\tt{2}} ~ \cdots ~ \mat{M}_{\tt{n}} \end{bmatrix},
    \end{equation}
    where as defined in definition \ref{def.sub_tns_env_mat}, $\mat{M}_{\tt{j}}$ denotes the environment matrix of the sub-network of $\tt{j}$.
\end{definition}


\begin{example} \label{ex.defEM}
Let us use a simple example to illustrate these definitions. 
Suppose $\T = (G,~A,~B,~C,~D,~E)$ is the tensor network shown in Figure \ref{fig.ex28}(a). 
The outgoing legs are marked in red.
It corresponds to the graph with $\set{V} = \{A, B, C, D, E\}$ and edges including $a = (A,B),~b = (B,C),~g = (C,E),~p = (B,B),~q = (B,B)$ (since the graph is not simple, using $(u,v)$ to index edges is not well-defined anymore. Here we use it to illustrate the idea) and so on.
As an example, $\set{U}_B = \{p,~q\}$, and $\set{C}_D = \
\{c,~d,~f\}$. 
$\E_E$, the environment tensor of sub-network $E$, is given on the left-hand side of Figure \ref{fig.ex28}(b).
The matricization $\mat N$ of $\E_E$, mentioned in definition \ref{def.sub_tns_env_mat}, is formed as the right-hand side of Figure \ref{fig.ex28}(b). 
By definition \ref{def.env_mat}, $\mat{M}_E$ is then given by $\mat N \otimes \mat{I}_s$, where $\mat{I}_s$ is the identity matrix with dimension the same as leg $s$. 

\begin{figure}[h]
\begin{subfigure}[Tensor Network]
{
\centering
\includegraphics[width=0.27\textwidth, height = 4.2cm]{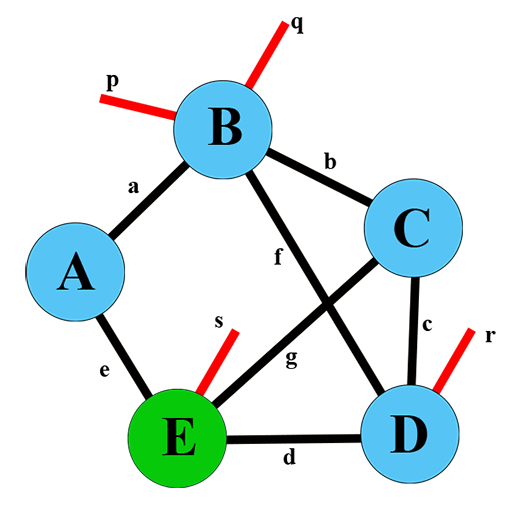}
\label{fig.ex28a}
}
\end{subfigure}
\begin{subfigure}[Environment Matrix]
{
\centering
\includegraphics[width=0.67\textwidth, height = 4.2cm]{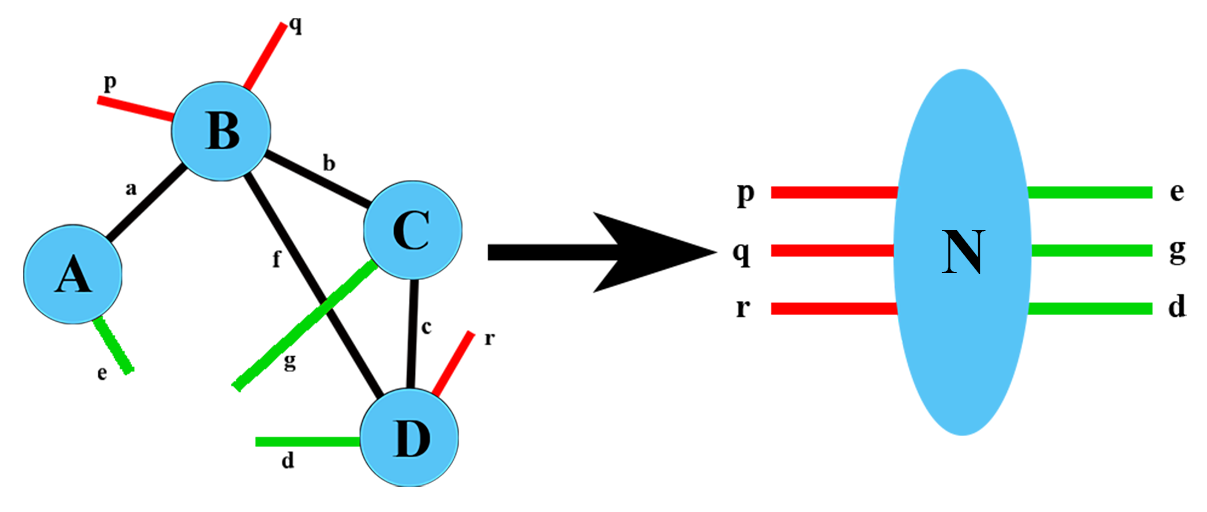}
\label{fig.ex28b}
}
\end{subfigure}
\captionlistentry{}
\label{fig.ex28}
\caption*{\smaller{Fig.\ref{fig.ex28}. Depictions of example tensor network with some contracted legs (edges between a pair of vertices) and uncontracted legs (edges adjacent to a single vertex), as well as formation of an environment matrix for one of the sites.}}

\end{figure}
\end{example}

\begin{definition}[Sitewise perturbation]
Let $\T = (G,\tt{1},~\tt{2},~\ldots,~\tt{n})$ be a tensor network. We say that $\hat{\tsr T} = \sT(\tt{1} + \tsr{\delta}^{(1)},~\tt{2} + \tsr \delta^{(2)}, ~\ldots,~ \tt{n} + \tsr \delta^{(n)})$ is the sitewise perturbed tensor network by tensors $(\tsr \delta^{(i)})_{i = 1}^n$. We call $(\tsr \delta^{(i)})_{i = 1}^n$ the sitewise perturbation. When no confusion is made, we call for brevity $\hat{\tsr{T}}$ the perturbed tensor network and $(\tsr \delta^{(i)})_{i = 1}^n$ the perturbation.
\end{definition}

In most cases, a relative perturbation best characterizes the error in tensor network algorithms like DMRG. 
To avoid exponentially growth in the bond dimension, some low rank approximation of tensors in MPS is usually needed, and thus the dominating error in representing a state by MPS is often caused by the truncation. 
The truncation strategy is usually based on singular values of the tensor to be compressed, and the error is correspondingly bounded relatively in Frobenius norms.

\begin{definition} [$\epsilon$-perturbation]
    \sloppy
    An $\epsilon$-perturbation to tensor network $\T = (G,\tt{1},\tt{2},\\ \ldots,\tt{n})$ is a sitewise perturbation $(\tsr{\delta}^{(i)})_{i = 1}^n$ such that $\norm{F}{\tsr{\delta}^{(i)}} \leq \epsilon \norm{F}{\tt{i}}$ for all $i$.
\end{definition}


To formally define error measures, consider tensors $\tt{1},~ \tt{2},~ \ldots ,~\tt{n}$ and let $\tsr T = \sT(\tt{1},~ \tt{2},~ \ldots ,~\tt{n})$.
Let $\tn{\delta} = (\tsr{\delta}^{(i)})_{i = 1}^n$ be a sitewise perturbation. Let $\hat{\tsr T} = \sT((\tt{1} + \tsr{\delta}^{(1)}),~ (\tt{2} + \tsr{\delta}^{(2)}),~ \ldots ,~ (\tt{n} + \tsr{\delta}^{(n)}))$ be the perturbed tensor network. We will use
$$
     \sE_a(\tn T, \tn{\delta}) := \norm{F}{\hat{\tsr T} -\tsr T}
$$
to denote the absolute error, and use
$$
    \sE_r(\tn T, \tn \delta) := \frac{\norm{F}{\hat{\tsr{T}} -\tsr{T}}}{\norm{F}{\tsr{T}}}
$$
for the relative error. Here $\norm{F}{\cdot}$ denotes the Frobenius or Hilbert-Schmidt norm.
In our discussion, the error is often controlled by $\norm{F}{\tsr{\delta}^{(i)}} \leq \epsilon_i$, for some positive small constant $0 \leq \epsilon_i \ll 1$.
A control on relative error also falls in this category by setting $\epsilon_i = \epsilon \norm{F}{\tt i}$ for some $0 \leq \epsilon \ll 1$. We will specify these controlling conditions whenever we introduce the perturbation. 


We primarily look at tight worst-case error bounds. The best possible worst-case error bound is the uniformly tight bound, as defined below. 
%

\begin{definition}[Uniformly tight bound]
Let $\X$ be a tensor network and $\S$ be the collection of all possible perturbation to $\X$. 
Let $f(\X, \tn{\delta})$ be a measure of error for $\tn{\delta} \in \S$ and $g(\X)$ be an error bound. 
We say $g$ is the uniformly tight bound if $\sup_{\tn{\delta} \in \S} f(\X, \tn{\delta}) = g(\X)$ holds for all $\X$.
\end{definition}

In our discussion, we will often need CP rank-1 tensors representing a product state to show tightness of bounds. We define it here. 

\begin{definition} [Product state network]
    Let $\tn T = (G,\tt 1,~ \tt 2,~ \ldots ,~\tt n)$ with real tensors $\tt i$ and $G=(\set{V}, \set{E})$. $\tn T$ is a product state network if
    there exists a set of unit vectors $\{\Bu_e\}_{e\in\set{E}}$ such that 
    for $\tt{j}$ at vertex $v_j$, $\tt{j} = \norm{F}{\tt{j}} \bigotimes_{e\in \set{E}} \Bu_{e}$.
\end{definition}

This definition can be extended to the complex case with appropriate use of conjugation. For each edge $e = (u, v)$ involving node tensors $\tt u$ and $\tt v$, one node tensor is composed with $\Bu_e$ in the tensor product while the other has $\Bu_e^\dagger$. Thus in the following discussion we do not distinguish between real and complex product state network.

%
%
%
%
%
%


\section{Worst-Case Perturbation Bounds for Tensor Networks} \label{s.wrst}

The conditioning of a tensor network $\tn T$ with graph $G$ can be regarded as the conditioning of the network-forming function $\sT_G$ involving independent variables $\tt{1},~\tt{2},~\cdots,~\tt{n}$. One natural measure of the stability of a tensor network is the condition number of function $\sT$ at $(\tt{1},~\tt{2},~\ldots,\tt{n})$. 

\begin{definition} [Tensor network absolute condition number] \label{d.stab_a}
    We define the relative condition number of a tensor network $\tn T = (G,\tt{1},~\tt{2},~\cdots,~\tt{n})$ as
    \begin{equation}
        \kappa_a(\tn T) := \lim_{\epsilon \rightarrow 0^+} \sup_{\tn{\delta} \in \S_\epsilon} \frac{\sE_a(\T,\tn{\delta})}{\epsilon},
    \end{equation}
    where $\S_\epsilon = \{\text{sitewise perturbations } (\tsr{\delta}^{(i)})_{i = 1}^n ~|~  \sum_i \norm{F}{\tsr{\delta}^{(i)}} \leq \epsilon$\}.
\end{definition}
\begin{definition} [Tensor network relative condition number] \label{d.stab_a}
    The relative condition number of a tensor network $\tn T = (G,\tt{1},~\tt{2},~\cdots,~\tt{n})$ representing $\tsr T = \sT(\tt{1},~\tt{2},~\cdots,~\tt{n})$ is
    \begin{equation}
        \kappa_r(\tn{T}) := \frac{\sum_{i=1}^n\norm{F}{\tt{i}}}{\norm{F}{\tsr T}}\kappa_a(\tn{T}) = \frac{\sum_{i=1}^n\norm{F}{\tt{i}}}{\norm{F}{\tsr T}}\lim_{\epsilon \rightarrow 0^+} \sup_{\tn{\delta} \in \S_\epsilon} \frac{\sE_a(\tn{T},\tn{\delta})}{\epsilon},
    \end{equation}
    where $\S_\epsilon = \{\text{sitewise perturbations } (\tsr{\delta}^{(i)})_{i = 1}^n ~|~  \sum_i \norm{F}{\tsr{\delta}^{(i)}} \leq \epsilon\}$.
\end{definition}

\begin{remark}
For now we always assume all node tensors $\tt{1},~\tt{2},~\cdots,~\tt{n}$ in the network are independent, i.e. they can be perturbed independently. It is possible to generalize the following approach to tensor networks involving dependent node tensors. See remark \ref{rmk.depend_tns} for detail.
\end{remark}

\begin{lemma} [Conditioning of tensor networks] \label{l.stab}
    Let $\tsr T = \sT_G(\tt{1},~\tt{2},~\cdots,~\tt{n})$, then $\kappa_a(\tn{T}) = \max_{i}\norm{2}{\mat{M}_{\tt{i}}}$ and
 $\kappa_r(\tn{T}) = \frac{\sum_{i=1}^n\norm{F}{\tt{i}}}{\norm{F}{\tsr T}}\max_{i}\norm{2}{\mat{M}_{\tt{i}}}$.
\end{lemma} 
\begin{proof}
For a tensor node $	\tt{j}$, in view of \eqref{eq.matvec_repr}, one has
\begin{equation}
    \pd{\vec({\tsr{T}})}{\vec(	\tt{j})} = \mat{M}_{\tt{j}}.
\end{equation}
Thus if we view the arguments $\tt{1},~\tt{2},~\cdots,\tt{n}$ as a long vector $\Bu$ (vectorize each $	t{j}$ and concatenate them), then the Jacobian is
\begin{equation}
    \mathbf{J}_{\vec({\tsr{T}})}(\Bu) = \begin{bmatrix} \mat{M}_{\tt{1}} ~ \mat{M}_{\tt{2}} ~ \cdots ~ \mat{M}_{\tt{n}} \end{bmatrix} = \mat{M}_{\tn{T}}.
\end{equation}
Similarly, treat the perturbation $(\tsr{\delta}^{(i)})_{i = 1}^n$ as a long vector $\boldsymbol{\tsr{\delta}}$, then $\sE_a(\tn{T},\tn{\delta}) = \sE_a(\vec({\tsr{T}}), \boldsymbol{\tsr{\delta}})$, and thus
\begin{equation}
    \kappa_a(\tn{T}) 
    = \lim_{\epsilon \rightarrow 0^+} \sup_{\sum_{i=1}^n\norm{F}{\tsr{\delta}_i} \leq \epsilon} 
    \frac{
    \norm{2}{\mat{M}_{\tn{T}} \cdot \boldsymbol{\tsr{\delta}}} + \O\left(\norm{2}{\boldsymbol{\tsr{\delta}}}^2\right)
    }{
    \epsilon
    }
    = \max_i \norm{2}{\mat{M}_{\tt{i}}}
\end{equation}
which completes the proof.
\end{proof}

With this lemma, now we state and prove the results on worst-case error for a general tensor network.

\begin{theorem} [Worst case absolute error for TN] \label{thm.worst_tn}
    Let $\tn T = (G,\tt{1},~\tt{2},~\cdots,~\tt{n})$ be a tensor network.
   Let $\epsilon_i \geq 0 \FORAL 1 \leq i \leq n$ and $\epsilon = \sum_i \, \epsilon_i$. Let the number of entries of $\tt{i}$ be $n_i$. Set $s_i = \sum_{j = 1}^{i-1}n_i$. Suppose ${\tn{T}}$ is perturbed by $(\tsr{\delta}^{(i)})_{i = 1}^n$, where $\norm{F}{\tsr{\delta}^{(i)}} \leq \epsilon_i \FORAL i$. Then up to an error of $\O(\epsilon^2)$, the worst case error $\sup_{\tn{\delta}} \sE_a(\tn{T},\tn{\delta})$ is the solution to the optimization problem
    \begin{equation} \label{eq.unif_bnd}
        \begin{aligned}
            &\max_{{\delta}} \norm{2}{\mat{M}_{\tn{T}} \cdot \vcr{\delta}} 
             \ \text{with} \  \norm{2}{\vcr{\delta}_{(s_i + 1, s_{i+1})}} = \epsilon_i, ~ \forall i,
        \end{aligned}
    \end{equation}
    where $\vcr{\delta}_{(s_i + 1, s_{i+1})}$ is the sub-vector formed by the $s_i + 1$ to $s_{i + 1}$th entry of $\vcr{\delta}$. In particular, an explicit tight bound is  
    \begin{equation} \label{eq.upper_bnd}
        \sE_a(\tn{T},\tn{\delta}) \leq \sum_{i = 1}^n \epsilon_i \norm{2}{\mat{M}_{\tt{i}}} + \O(\epsilon^2).
    \end{equation}
    When ${\tsr{T}}$ contracts to a scalar, the bound \eqref{eq.upper_bnd} is uniformly tight.
    Further, $\sE_a(\tn{T},\tn{\delta})= \kappa_a(\tn T)\epsilon + \O(\epsilon^2)$.
\end{theorem}

\begin{corollary} [Worst case relative error for TN] \label{thm.worst_tn}
    Let $\tn T = (G,\tt{1},~\tt{2},~\cdots,~\tt{n})$ be a tensor network.
  Suppose ${\tn{T}}$ is perturbed by $(\tsr{\delta}^{(i)})_{i = 1}^n$, where $\sum_{i=1}^n \norm{F}{\tsr{\delta}^{(i)}}/\norm{F}{\tt{i}} \leq \epsilon$. The relative error in the tensor network state is $\sE_r(\tn{T},\tn{\delta}) \leq \kappa_r(\tn{T})\epsilon + \O(\epsilon^2)$.
\end{corollary}


\begin{proof}
As in the proof above, we regard the perturbation $(\tsr{\delta}^{(i)})_{i = 1}^n$ as a long vector $\vcr{\delta}$. Then from the proof of lemma~\ref{l.stab}, 
\begin{equation}
    \label{eq:norm_pert}
    \sE_a({\tn{T}}, \tn{\delta}) = \norm{2}{\mat{M}_{\tn{T}} \cdot {\vcr{\delta}}} + \O\left(\norm{2}{\vcr{\delta}}^2\right).
\end{equation}
Let $\Ba_i = \vcr{\delta}_{(s_i + 1, s_{i+1})} = \vec(\tsr{\delta}^{(i)})$.
The equality constraint $\norm{F}{\tsr{\delta}^{(i)}} = \epsilon_i$ in \eqref{eq.unif_bnd} holds when optimizing over any $\tsr{\delta}^{(i)}$ with $\norm{F}{\tsr{\delta}^{(i)}} \leq \epsilon_i$ since
$\mat{M}_{\tn{T}} \cdot \vcr{\delta}$ is linear in $\Ba_i$.
Since $\mat{M}_{\tn{T}} \cdot \Bu = \sum_{i = 1}^n \mat{M}_{\tt{i}} \cdot \Bu_i$, the bound \eqref{eq.upper_bnd} follows from \eqref{eq.unif_bnd} as
\begin{equation} \label{eq.proof_of_worst}
    \norm{2}{\mat{M}_{\tn{T}} \cdot \vcr{\delta}} = \norm{2}{\sum_{i = 1}^n \mat{M}_{\tt{i}} \cdot \Ba_i} \leq \sum_{i = 1}^n \norm{2}{\mat{M}_{\tt{i}}} \cdot \norm{F}{\tsr{\delta}^{(i)}} \leq \sum_{i = 1}^n \epsilon_i \norm{2}{\mat{M}_{\tt{i}}}.
\end{equation}

The second inequality 
is tight for some $\tsr{\delta}^{(1)},\ldots,\tsr{\delta}^{(n)}$ given any choice of tensor network ${\tn{T}}$. 
The first inequality is tight if $\mat{M}_{\tt{1}},\ldots,\mat{M}_{\tt{n}}$ share the same largest left singular vector. 
This holds when ${\tsr{T}}$ corresponds to a scalar, since then each $\mat{M}_{\tt{i}}$ is a row vector with a left singular vector $\Bu = [1]$.
In this case, we can always choose $\Ba_i = \norm{F}{\tsr{\delta}^{(i)}}  \mat{M}_{\tt{i}}^\dagger/\norm{F}{\mat{M}_{\tt{i}}}$ to make \eqref{eq.proof_of_worst} an equality.

Lastly, it is straight forward to check that when ${\tn{T}}$ is a product state network, the sitewise perturbation $\tsr{\delta}^{(i)} = \epsilon_i \frac{\tt{i}}{\norm{F}{\tt{i}}}$ achieves the bound.
\end{proof}


\begin{remark}
Without an additional assumption on properties of ${\tsr{T}}$, solving the optimization problem \eqref{eq.unif_bnd} is equivalent to solving a set of quadratic equations. 
As in the proof above, we write $\Ba_i = \vec(\tsr{\delta}^{(i)})$, then the KKT condition tells that the solution to \eqref{eq.unif_bnd} satisfies the quadratic system,
\begin{equation} \label{eq.kkt1}
    {\mat{M}_{\tn{T}}}^\dagger \mat{M}_{\tn{T}} \begin{pmatrix} \Ba_1 \\ \Ba_2 \\ \vdots \\ \Ba_n \end{pmatrix} = \begin{pmatrix} \mu_1\Ba_1 \\ \mu_2\Ba_2 \\ \vdots \\ \mu_n\Ba_n \end{pmatrix},
\end{equation}
\begin{equation} \label{eq.kkt2}
    \norm{2}{\Ba_i}^2 = \epsilon_i^2,~i = 1,~2,~\cdots,~n,
\end{equation}
where $\Ba_i$ and $\mu_i \geq 0$ are unknowns to be solved.
We know the solution to this system exists since solution to \eqref{eq.unif_bnd} exists. 
Left multiply $\begin{pmatrix} \Ba_1^\dagger~\Ba_2^\dagger~\cdots~\Ba_n^\dagger \end{pmatrix}$ on both sides of \eqref{eq.kkt1} and substitute in \eqref{eq.kkt2}, the worst case error is then given by $\sup_{\tn {\delta}} \sE_a(\tn{T},\tn{\delta}) = \sqrt{\sum_{j = 1}^n \mu_j \epsilon_j^2}$ to order $\O(\epsilon^2)$. 
\end{remark}


\begin{corollary} [Error of one-site perturbation] \label{cor.one_site}
Let $\tn T = (G,\tt{1},~\tt{2},~\cdots,~\tt{n})$ be a tensor network. Let $\epsilon > 0$. 
Suppose $\tt{1}$ is perturbed as $\tsr{\hat{T}}{}^{(1)} = \tt{1} + \tsr{\delta}$ for some $\norm{F}{\tsr{\delta}} \leq \epsilon$. Let $\hat{{\tsr{T}}} = \sT(\tsr{\hat{T}}{}^{(1)},~\tt{2},~\cdots,~\tt{n})$. 
Then $\norm{F}{\hat{{\tsr{T}}} - {\tsr{T}}} \leq \epsilon \cdot \norm{2}{\mat{M}_{\tt{1}}}$. 
In particular if $\norm{F}{\tsr{\delta}} \leq \epsilon \norm{F}{\tt{1}}$, then the relative error $\frac{\norm{F}{\hat{{\tsr{T}}} - {\tsr{T}}}}{\norm{F}{{\tsr{T}}}} \leq \epsilon \cdot \frac{\norm{2}{\mat{M}_{\tt{1}}} \norm{F}{\tt{1}}}{ \norm{F}{{\tsr{T}}}}$.
The bound is uniformly tight and bounded above by $\kappa_r(\tn{T})\epsilon$. 
The right-hand side is minimized as $\epsilon$ if ${\tsr{T}}$ is in a canonical form centered at $\tt{1}$.
\end{corollary}

Most of the proof for this corollary follows from the proof of theorem \ref{thm.worst_tn}. 
One subtle point is the uniform tightness -- since only one site is perturbed, the environment matrix of the tensor network consists of only one block $\mat{M}_{\tt{1}}$. 
Thus the argument holds trivially that $\mat{M}_{\tt{j}}$ share the same largest left singular vector, and hence the uniform tightness of the first inequality. When first inequality is made an equality 

\begin{remark} \label{rmk.depend_tns}
To analyze tensor networks with dependent nodes, we still construct the environment matrix, but instead of allowing arbitrary perturbation that satisfies the norm condition, say $\norm{F}{\tsr{\delta}^{(i)}} \leq \epsilon_i$, we need to pass the dependency between nodes to perturbations on those nodes. 
For example if $\tt{2} = \tt{1}{}^\dagger$, then we must have $\tsr{\delta}_2 = \tsr{\delta}^{(1)}{}^\dagger$. Then to the leading order, the procedure and result of theorem \ref{thm.worst_tn} still holds.
\end{remark}

\section{Average-Case Perturbation Bounds for Tensor Networks} \label{s.ave}

Using the approach from Section \ref{s.wrst}, we can analyze the average case error caused by relative sitewise perturbations.
In particular, we consider the expected value of the error in the overall state given randomly generated perturbations to tensor network sites.
To do this, we need the aid of the entrywise normalized version of the tensor network defined below. Intuitively, we scale $\tn{T}$ to $\overline{\tn{T}}$ making each entry of node tensors $\tt{j}$ of magnitude 1 on average.

\begin{definition} [Entrywise normalized tensor network]
\label{def:sntn}
The entrywise normalized tensor network is defined as $\overline{\tn{T}} = \sT\left(\tsr{\overline{T}}^{(1)},~\tsr{\overline{T}}^{(2)},~\ldots,~\tsr{\overline{T}}^{(n)}\right)$, where $\tsr{\overline{T}}^{(j)} = \frac{\sqrt{N_j} \, \tt{j}}{\norm{F}{\tt{j}}}$, and $N_j$ is the number of entries in $\tt{j}$.
\end{definition}


\begin{theorem} [Average-case error] \label{thm.ave}
    Let $\tn{T} = \sT(\tt{1},~\tt{2},~\cdots,~\tt{n})$ be a tensor network. Let $\epsilon > 0$. Suppose $\tn{T}$ is perturbed by random perturbations $(\tsr{\delta}^{(i)})_{i = 1}^n$, where entries of $\tsr{\delta}^{(i)}$ are of mean 0 and the same variance such that $\bE\norm{F}{\tsr{\delta}^{(i)}}^2 \leq \epsilon^2 \norm{F}{\tt{i}}^2$ for all $i$, and all entries of $\tsr{\delta}$ are independent. Then 
    \begin{equation} \label{eq.ave_err1}
        \bE\sE_r(\tn{T},\tn{\delta})^2 \leq \epsilon^2 \frac{
            \norm{F}{\mat{M}_{\tsr{\overline{T}}}}^2
        }{
            \norm{F}{\tsr{\overline{T}}}^2
        }
         + \O(\epsilon^3). 
    \end{equation}
    Equality in \eqref{eq.ave_err1} is achieved if $\bE\norm{F}{\tsr{\delta}^{(i)}}^2 = \epsilon^2 \norm{F}{\tt{i}}^2 \FORAL i$. If in fact all entries of $(\tsr{\delta}^{(i)})_{i = 1}^n$ have the same variance $\sigma^2 \ll 1$, then
    \begin{equation} \label{eq.ave_err2}
        \bE\sE_r(\tn{T},\tn{\delta})^2 
        =
        \sigma^2 \frac{
            \norm{F}{\mat{M}_{\tsr{T}}}^2
        }{
            \norm{F}{\tsr{T}}^2
        }
         + \O(\sigma^3).
    \end{equation}
\end{theorem}

\begin{proof}
The expected error is maximized when variance of each entry is maximized, so we focus on showing that
    \begin{equation} \label{eq.ave_err1}
        \bE\sE_r(\tn{T},\tn{\delta})^2 = \epsilon^2 \frac{
            \norm{F}{\mat{M}_{\tsr{\overline{T}}}}^2
        }{
            \norm{F}{\tsr{\overline{T}}}^2
        }
         + \O(\epsilon^3). 
    \end{equation}
when $\bE\norm{F}{\tsr{\delta}^{(i)}}^2 = \epsilon^2 \norm{F}{\tt{i}}^2 \FORAL i$.
We show that the expected error of this perturbation is equivalent to the expected error due to a $\tsr{\overline{\delta}} = (\tsr{\overline{\delta}}^{(i)})_{i = 1}^n$ random sitewise perturbation to the entrywise normalized tensor network $\overline{\tn{T}}$ (definition~\ref{def:sntn}) with independent entries of variance $\epsilon^2$. 
In particular, for all $\tsr{\delta}^{(i)}$, set $\tsr{\overline{\delta}}^{(i)} = \frac{\sqrt{N_i}}{\norm{F}{\tt{i}}} \tsr{\delta}^{(i)}$. 
Since $\bE \norm{F}{\tsr{\delta}^{(i)}}^2 = \epsilon^2 \norm{F}{\tt{i}}^2$, 
we get $\bE \norm{F}{\tsr{\overline{\delta}}^{(i)}}^2 = \epsilon^2 N_i = \epsilon^2 \norm{F}{\tsr{\overline{T}}^{(i)}}^2$, so the variance of each entry in $\tsr{\delta}^{(i)}$, for each $i$, is $\epsilon^2$.
Since $\overline{\tn{T}}$ is obtained by scaling each input $\tt{i}$ of $\tn{T}$ by $\frac{\sqrt{N_i}}{\norm{F}{\tt{i}}}$, and $\tsr{\overline{\delta}}$ is also obtained by scaling each node $\tsr{\delta}^{(i)}$ of $\tsr{\delta}$ in the same way, we see by proposition \ref{prop.lin} that $\sE_r(\tsr{T},\tn{\delta}) = \sE_r(\tsr{\overline{T}}, \tsr{\overline{\delta}})$.

It now suffices to consider the entrywise normalized tensor network perturbation.
Treat the perturbation $\tn{\overline{\delta}}=(\tsr{\bar{\delta}}^{(i)})_{i = 1}^n$ as a long vector $\vcr{\tsr{\overline{\delta}}}$. 
We have (via \eqref{eq:norm_pert}),
\begin{equation} \label{eq.exp_comp}
    \begin{aligned}
    \bE \sE_a(\overline{\tn{T}}, \overline{\tn{\delta}})^2 
    &= 
    \bE \left[\norm{2}{\mat{M}_{\tsr{\overline{T}}} \vcr{\tsr{\overline{\delta}}}} + \O(\norm{2}{\vcr{\tsr{\overline{\delta}}}}^2)\right]^2 
    = 
    \bE \norm{2}{\mat{M}_{\tsr{\overline{T}}} \vcr{\tsr{\overline{\delta}}}}^2 + \O(\norm{2}{\vcr{\tsr{\overline{\delta}}}}^3) \\
    &=
    \sum_{i = 1}^n \bE \norm{F}{\mat{M}_{\tsr{\overline{T}}^{(i)}} \vcr{\tsr{\overline{\delta}}}^{(i)}}^2 + \O(\epsilon^3) 
    =
    \epsilon^2 \cdot \sum_{i = 1}^n \norm{F}{\mat{M}_{\tsr{\overline{T}}^{(i)}}}^2 + \O(\epsilon^3) \\
    &= 
    \epsilon^2 \cdot \norm{F}{\mat{M}_{\tsr{\overline{T}}}}^2 + \O(\epsilon^3),
    \end{aligned}
\end{equation}
where in the second line we used independence between entries of $\tsr{\overline{\delta}}$ and the fact that their mean is 0. 
Thus $\bE \sE_r(\tn{T},\tn{\delta})^2 = \bE \sE_r(\overline{\tn{T}}, \tsr{\overline{\delta}})^2 = \epsilon^2 \frac{\norm{F}{\mat{M}_{\tsr{\overline{T}}}}^2}{\norm{F}{\tsr{\overline{T}}}^2} + \O(\epsilon^3)$. 
When the variance is uniformly bounded by $\sigma^2$, then the computation of \eqref{eq.exp_comp} gives \eqref{eq.ave_err2} without the need to rescale $\tn{T}$ and $\tsr{\delta}$.
\end{proof}

\begin{remark}
It is possible to give a concentration of the distribution of $\sE_r(\tn{T},\tn{\delta})$ given the information of distribution of $\tsr{\delta}$. 
For example if each entry of $\tsr{\delta}$ is sub-Gaussian with variance parameter $v$, as in the case of bounded perturbation or Gaussian perturbation, 
then an application of Hanson-Wright type inequality shows that $\sE_a(\tn{T}, \tn{\delta})$ is also sub-Gaussian with variance parameter a constant multiple of $v\norm{2}{\mat{M}_{\tsr{T}}}^2$. 
See theorem 2.1 in \cite{rudelson2013HWineq} for detail. 
\end{remark}





\section{Error in Tensor Network Methods} \label{s.stab_met}

The perturbation bounds in Section~\ref{s.wrst} can describe the affect of the choice of tensor network gauge on the worst-case error amplification in tensor network methods.
For attainable accuracy in optimization methods, we bound the worst-case amplification truncation error in the representation of exact solutions by the condition number of the tensor network using all-site perturbation bounds.
These bounds can also be extended to use the average case analysis in Section~\ref{s.ave}.
Additionally, we describe how our theory can be applied to simulation of quantum systems via time-evolution of tensor networks by application of local operators.

\subsection{Attainable Accuracy in Optimization} \label{ss.dmrg}

The choice of gauge (tensor network representation) can bound the attainable accuracy in optimization procedures, such as DMRG for the eigenvalue problem~\cite{RevModPhys.77.259,white1992density,white1993density} or solutions to linear systems via alternating least squares~\cite{holtz2012alternating,dolgov2014alternating,dolgov2013tt,oseledets2012solution,yu2017finding,khemani2016obtaining,tobler2012low,oseledets2011dmrg}.
We approximately bound the attainable accuracy by considering the effect of sitewise truncation on any particular gauge representing the exact eigenstate.
In particular, we consider a truncation procedure that discards singular values of $\tsr{T}^{(i)}$ sequentially for $i = 1,2,\ldots,n$, up to an error of approximately $\varepsilon_i \|\tsr{T}^{(i)}\|_F$ for some $0 \leq \varepsilon_i \ll 1$.
\begin{corollary} [Worst case attainable accuracy] \label{thm:attain_acc}
Consider a reference state $\vcr x$.
After truncation of each site to relative accuracy $\varepsilon_i\leq \varepsilon$ of a tensor network representing $\vcr x$, the worst-case attainable accuracy for the best possible gauge is
\[\frac{\|\vcr{\hat{x}} - \vcr{x}\|_2}{\|\vcr x\|_2} =  \min_{\substack{\tsr{T}^{(1)},  \tsr{T}^{(2)},  \ldots, \tsr{T}^{(n)}, \\
\vcr x  = \vec(\sT(\tsr{T}^{(1)}, ~ \tsr{T}^{(2)}, ~ \ldots ~ , \tsr{T}^{(n)}))}} \sum_{i=1}^n  \varepsilon_i \frac{\norm{2}{\mat{M}_{\tsr{T}^{(i)}}} \cdot \norm{F}{\tsr{T}^{(i)}}} {\norm{2}{\vcr x}} + \O(\varepsilon^2).\]  
\end{corollary}
\begin{proof}
For any $\sT(\tsr{T}^{(1)}, ~ \tsr{T}^{(2)}, ~ \ldots ~ , \tsr{T}^{(n)})$, the truncations yield perturbations $\tsr{\delta}^{(1)},\ldots,\tsr{\delta}^{(n)}$ with $\|\tsr{\delta}^{(i)}\|_F \leq \varepsilon\|\tsr{T}^{(i)}\|_F$.
Consequently, we can apply theorem~\ref{thm.worst_tn} to obtain the worst case absolute error,
    \begin{equation}
        \sE_a(\T, \delta) \leq \sum_{i = 1}^n \varepsilon_i \norm{F}{\tsr{T}^{(i)}} \norm{2}{\mat{M}_{\tsr{T}^{(i)}}} + \O(\varepsilon^2).
    \end{equation}
The relative error bound follows by dividing by $\norm{2}{\vcr x}$ and minimizing overall possible gauges.
\end{proof}
This attainable accuracy bound includes all possible canonical forms for tensor networks representing $\vcr x$.
However, a lower error is attainable by transforming the tensor network to a canonical form centered at the $i$th site before truncation of that site, as done in DMRG.
The following theorem bounds the error in this scenario.
\begin{theorem} [Worst case attainable accuracy with canonicalization] \label{thm:attain_acc_can}
Consider a reference state $\vcr x$ and any tensor network representation thereof, $\tsr{T}^{(1)},~ \tsr{T}^{(2)},~\ldots,~\tsr{T}^{(n)}$, such that $\vcr x  = \vec(\sT(\tsr{T}^{(1)}, ~ \tsr{T}^{(2)}, ~ \ldots ~ , \tsr{T}^{(n)}))$.
If we truncate each site to relative accuracy $\varepsilon_i\leq \varepsilon$ after putting the tensor network into a canonical form with that site is the center, we obtain  $\vcr{\hat{x}}=\vec(\sT(\tsr C^{(1)}, ~ \tsr C^{(2)}, ~ \ldots ~ , \tsr{C}^{(n)}))$ with error,
\[\frac{\|\vcr{\hat{x}} - \vcr{x}\|_2}{\|\vcr x\|_2} \leq  n\varepsilon + \O(\varepsilon^2)\]
\end{theorem}
\begin{proof}
Let $\vcr x^{(i)}$ be the state after the $i$th truncation, so that $\vcr x^{(0)}=\vcr x$ and $\vcr x^{(n)} = \vcr{\hat{x}}$.
We can make use of corollary~\ref{cor.one_site} to bound the relative error of each truncation by
\[\norm{2}{\vcr x^{(i+1)} - \vcr x^{(i)}} \leq \varepsilon_i\norm{2}{\vcr x^{(i)}}.\]
Further, $\vcr{\hat{x}} - \vcr{x} = \sum_{i=0}^{n-1}\vcr x^{(i+1)} - \vcr x^{(i)}$, so 
$\norm{2}{\vcr{\hat{x}} - \vcr{x}} \leq \sum_{i=1}^n\varepsilon_i\norm{2}{\vcr x^{(i)}} \leq n\varepsilon \norm{2}{\vcr x} + \O(\varepsilon^2)$.
\end{proof}


Given a perturbed normalized eigenvector $\hat{\vcr x} = \vcr x + \vcr \delta$, we can bound the error in the energy (eigenvalue) $E$ of Hamiltonian $\mat H$.
Let $\vcr \nu$ be the component of $\vcr \delta$ that is perpendicular to $\vcr x$, so $\|\vcr \nu\|_2\leq \|\vcr \delta\|_2$ and $\langle \vcr x,\vcr{\nu}\rangle=0$, then
\[\langle \vcr{\hat{x}}, \mat H \vcr{\hat{x}}\rangle = (1-\|\vcr{\nu}\|_2^2)E + \langle \vcr{\nu}, \mat H\vcr{\nu}\rangle.\]
Consequently, the error scales quadratically with the magnitude of the perturbation,
\[|\hat{E}-E|\leq \|\vcr \nu\|_2^2|E| + \|\mat H\|_2\|\vcr \nu\|^2 \leq\|\vcr \delta\|_2^2|E| + \|\mat H\|_2\|\vcr \delta\|^2.\]

\subsection{Tensor Network Time Evolution} \label{ss.qcirc}


Many tensor network methods can be described by application of a series of sequence of operators that act on nearby sites in a tensor network and the result is approximated to some accuracy via truncation of bonds between those sites.
This approach can be used for simulation of quantum circuits~\cite{markov2008simulating,2017arXiv171005867P,jozsa2006simulation,schuch2007computational,guo2019general}.
Further, in quantum simulation, tensor network time-evolution methods~\cite{orus2014practical,lubasch2014algorithms,lubasch2014unifying,haegeman2016unifying,czarnik2012projected,garcia2006time,orus2008infinite,vidal2007classical} are used to approximately compute $\vcr{x}(t) = e^{t\mat{A}}\vcr{x}(0)$ where the state $\vcr{x}(t)$ is represented by a tensor networks $\vcr{x}(t_i) \approx \vec(\sT(\tsr{T}^{(1)}_{i},\ldots,\tsr{T}^{(n)}_{i})$ where $t_i=(i-1)\tau$.
Each time step $\vcr{x}(t_{i+1}) = e^{\tau \mat A}\vcr{x}(t_i)$ is performed by a Trotter expansion of the time-evolution operator
\[e^{\tau \mat A}= e^{\tau \mat{A}^{(1)}}\cdots e^{\tau \mat{A}^{(N)}} + O(\tau^2),\]
where $\mat{A}=\sum_i \mat{A}^{(i)}$ and each $\mat{A}^{(i)}$ is a local operator.
Thus, a substep of the time-evolution is to apply each $e^{\tau \mat A^{(i)}}$ to the current state, which may involve updating a neighborhood of sites in a tensor network and truncating.
Such time-evolution methods with tensor networks are used to simulate dynamics of quantum systems described by Hamiltonian $\mat H=i\mat A$~\cite{orus2014practical} as well as to perform imaginary time evolution~\cite{haegeman2016unifying,czarnik2012projected,orus2008infinite,orus2014practical} to compute the dominant eigenvector of $\mat A$ (used to get the ground-state of Hamiltonian $\mat{H} = -\mat A$).

The error of each time-evolution substep corresponds to a single-site or multi-site perturbation and its amplification would be bounded via corollary~\ref{cor.one_site} (with $\varepsilon_i>0$ for the tensors updated) and minimized when the tensor network is a canonical form with the center corresponding to the site or sites being approximately updated.
However, it can be costly to move the center each time we apply a local operator.
As a solution, one could try to find an optimal order to apply the gates such that the movement cost is minimized or acceptable. 

%


\section{Error Bounds for 1D and 2D Tensor Networks} \label{s.appl}


We now apply the results in section \ref{s.wrst} to two widely used tensor networks -- MPS (1D) and PEPS (2D), as shown in Figure \ref{fig.mpspeps}(a) and (b) below. 
The uncontracted legs, as colored in red in Figure \ref{fig.mpspeps}(a) and \ref{fig.mpspeps}(b), are referred to as the \textit{physical legs} or \textit{outgoing legs}.
In the context of tensor decompositions, the MPS tensor network is often referred to as tensor train~\cite{oseledets2011tensor}.
\begin{figure}[h]
    \centering
    \minipage{0.50\textwidth}
      \centering
      \includegraphics[width=\textwidth, height = 4.5cm]{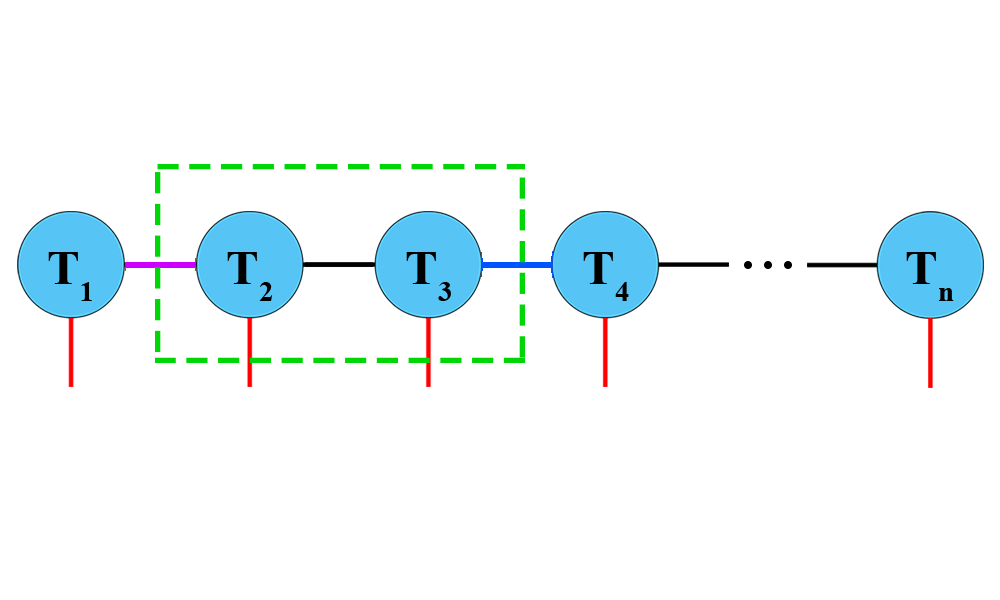}
      \caption*{\smaller{Fig.\ref{fig.mpspeps}(a). MPS with $n$ nodes}}
    \endminipage\hfill
    \minipage{0.50\textwidth}%
      \centering
      \includegraphics[width=\textwidth, height = 4.5cm]{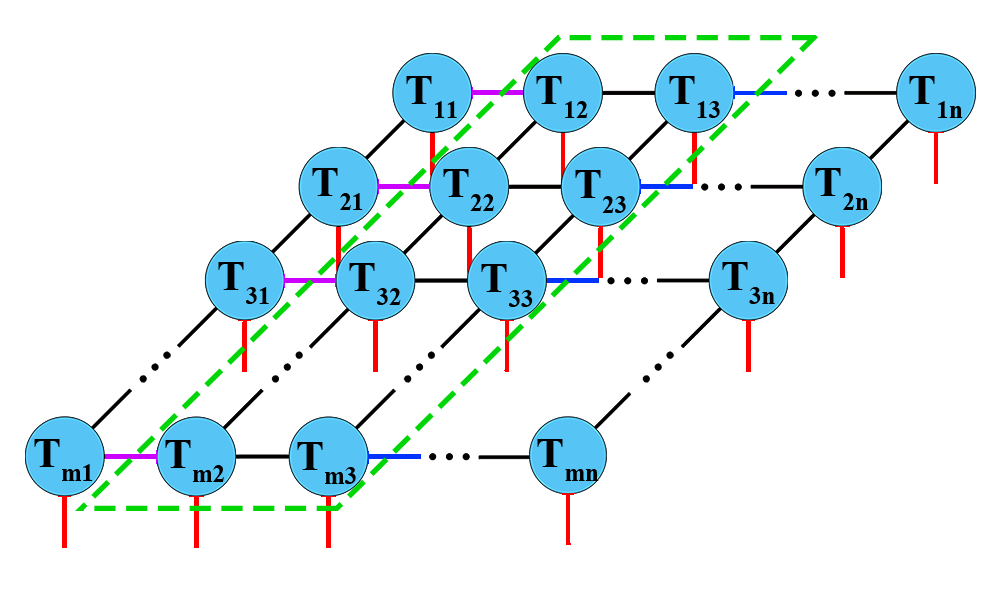}
      \caption*{\smaller{Fig.\ref{fig.mpspeps}(b). PEPS with $m \times n$ nodes}}
      \hfill
    \endminipage
    \captionlistentry{}
    \label{fig.mpspeps}
\end{figure}

We will derive worst-case tight bounds on sitewise perturbation errors for general and canonical MPS/PEPS. 
In addition, we also stress that tightness of the bounds also holds under the constraint that the tensors have a full rank matricization,
which is a realistic situation in practical computations and simulations. 
Since the bounds are tight, we can then compare them and conclude whether converting to canonical forms helps to improve stability. 

For these structured tensor networks, a simpler notation is available. 
With an MPS, we can still use the multiplication notation for tensor contractions, namely $\tsr A \tsr B$ means contracting the joint leg between $\tsr A \AND \tsr B$. 
For an MPS $(G,\tt{1},\, \tt{2},\, \cdots \, , \tt{n})$, we will use the notation $\tn{T}^{[m,\, n]} := (S[G, \{v_m,\cdots, v_n\}], \tt{m}, \tt{m+1}, \cdots, \tt{n})$, for $m \leq n$, and $\tn{T}^{[m,\, n]} = \mat I$ for $m > n$ (the dimension does not matter, as we will only use its 2-norm). 
This convention simplifies the expressions in our analysis.
We use $\mat{T}^{[m,n]}_\rightarrow$ to denote the matricization of $\tsr{T}^{[m,n]}$,
with the right-going leg of node $\tt{n}$ being column index (to be contracted) and others being row index.
Similarly, we also introduce the matricization $\mat{T}^{[m, n]}_\leftarrow$, where the column (contracted) index is the left-going leg of $\tt m$.
As an example, $\tn{T}^{[2, 3]}$ is boxed in Figure \ref{fig.mpspeps}(a), and $\mat{T}^{[2, 3]}_\rightarrow$ corresponds to the matricization of $\tsr{T}^{[2, 3]}$ with the red leg being the column leg, and others being the row legs. 
Similarly, the column leg for $\mat{T}^{[2, 3]}_\leftarrow$ is the purple leg in Figure \ref{fig.mpspeps}(a). 

For PEPS, we index node tensors at $i$th row and $j$th column by $\tsr T^{(i,j)}$ (with (1, 1) on the upper left corner, see Figure \ref{fig.mpspeps}(b)), 
use $\tn{T}^{(\cdot, j)}$ and $\tn{T}^{(i,\cdot)}$ to denote the $j$th column and $i$th row of PEPS, 
We denote $\mat{T}^{(\cdot, [p,q])}_\rightarrow$ to be the matricization of $\tsr{T}^{(\cdot, [p,q])}$ with all right-going legs in $\tsr{T}^{(\cdot, q)}$ being column (contracted) indices. 
Similarly we use arrows pointing left, up, and down to specify the orientation of matricization of a rectangular block of PEPS.
As an example, in Figure \ref{fig.mpspeps}(b), $\tsr{T}^{(\cdot,[2,3])}$ is boxed in green, and column legs for $\tsr{T}^{(\cdot,[2,3])}_\rightarrow$ is colored in red.

We discuss MPS and PEPS separately in Section \ref{ss.mps} and \ref{ss.peps}, and put special focus on stability of canonical forms and applications to DMRG-like algorithms and quantum circuit simulation in Sections \ref{ss.dmrg} and \ref{ss.qcirc}.

\subsection{Stability and error analysis of MPS} \label{ss.mps}
Now we compute the worst-case error bounds for MPS with respect to relative sitewise perturbations. We derive bounds for both general MPS and canonical MPS, prove these two bounds together, and then compare them.

\begin{corollary}[Single-site perturbation]
    Let $\tn{T} = (G, \tt{1}, ~ \tt{2}, ~ \cdots ~ , \tt{n})$ be an MPS. Suppose $\tn{T}$ is perturbed only at node $\tt j$ by $\tsr{\delta}^{(j)}$, 
    with $\norm{F}{\tsr{\delta}^{(j)}} \leq \epsilon \norm{F}{\tt{j}}$. Let the perturbed version be $\hat{\tn{T}}$. Then
    \begin{equation} \label{eq.mps_worst}
      \frac{\norm{F}{\tsr{T} - \hat{\tsr{T}}}}{\norm{F}{\tsr{T}}} \leq \epsilon\cdot \frac{
      \bnorm{2}{\mat{T}^{[1, j-1]}_\rightarrow} \cdot \norm{F}{\tt{j}} \cdot \bnorm{2}{\mat{T}^{[j+1, n]}_\leftarrow}
      }{
      \norm{F}{\tsr{T}}
      }  
      + \O(\epsilon^2).  
    \end{equation}
    If $\tn{T}$ is in a canonical form centered at $\tt{j}$, then 
    \begin{equation} \label{eq.mps_worst}
      \frac{\norm{F}{\tsr{T} - \hat{\tsr{T}}}}{\norm{F}{\tsr{T}}} \leq \epsilon 
      + \O(\epsilon^2).  
    \end{equation}
    Moreover, the bound are tight. Therefore canonical form reduces the worst-case error for single-site perturbation. 
\end{corollary}

\begin{proof}
    Apply corollary \ref{cor.one_site} to MPS $\tn{T}$. Note that the environment matrix for $\tt{j}$ is $\mat{T}^{[1, j-1]}_\rightarrow \otimes \mat{T}^{[j+1, n]}_\leftarrow$ when $j$ is not 1 or $n$. When $j$ is 1 or $n$, remove the left and right part in the tensor product accordingly.
\end{proof}

\begin{corollary} [All-site perturbation for general MPS] \label{cor.mps_worst}
Let $\tn{T} = (G,\tt{1}, ~ \tt{2}, ~ \cdots ~ , \tt{n})$ be an arbitrary MPS. Suppose $(\tsr{\delta}^{(i)})_{i = 1}^n$ is an $\epsilon$-perturbation to $\tn{T}$. Then
\begin{equation} \label{eq.mps_worst}
  \sE_r(\tn{T}, \delta) \leq \epsilon\cdot \sum_{j = 1}^n \frac{
  \bnorm{2}{\mat{T}^{[1, j-1]}_\rightarrow} \cdot \norm{F}{\tt{j}} \cdot \bnorm{2}{\mat{T}^{[j+1, n]}_\leftarrow}
  }{
  \norm{F}{\tsr{T}}
  }  
  + \O(\epsilon^2).  
\end{equation}
Moreover, the bound is tight.
\end{corollary}

\begin{corollary} [All-site perturbation for canonical MPS] \label{cor.cmps_worst}
Let $\tn C = (G,\, \tsr{C}^{(1)}, \tsr{C}^{(2)}, \cdots, \tsr C^{(n)})$ be an MPS representing $\tsr C$ in canonical form with bond dimension bounded by $D$, centered at site $n$. 
Suppose $(\tsr{\delta}^{(i)})_{i = 1}^n$ is an $\epsilon$-perturbation to $\tsr C$.
Then up to an error of $\O(\epsilon^2)$,
\begin{equation} \label{eq.cmps_worst}
    \sE_r(\tn C, \tn{\delta}) 
    \leq 
    \epsilon\cdot \left( 1 + \sum_{j = 1}^{n-1}
    \frac{  \sqrt{D_j} \bnorm{2}{\mat C^{[j+1, n]}_\leftarrow}  }{  \norm{F}{\tsr C}  }
    \right)
    \leq
    \epsilon\cdot \left( 1 + (n-1)\sqrt{D}\right).
\end{equation}
Moreover, the bounds are tight, given that the bond dimension is $D$ across the MPS.
\end{corollary}

\begin{proof} 
    We use theorem \ref{thm.worst_tn}. 
    Apply \eqref{eq.upper_bnd} to $\tn{T}$, using that $E_{\tt{j}} = \mat{T}^{[1,j-1]} \otimes \mat{I}_{D_j} \otimes \mat{T}^{[j+1,n]}$ for $2 \leq j \leq n-1$, and $\mat{M}_{\tt{1}} = \mat I_{D_1} \otimes \mat{T}^{[2,n]}$, $\mat{M}_{\tt{n}} = \mat{T}^{[1,n-1]} \otimes \mat I_{D_n}$, where $D_j$ represents the physical dimension at node $j$, we get \eqref{eq.mps_worst}. 
    One can then deduce \eqref{eq.cmps_worst} from \eqref{eq.mps_worst} by inserting $\norm{2}{\mat C^{[1,j-1]}_\rightarrow} = 1$ and $\norm{F}{\tsr{C}^{(j)}} = \sqrt{D_j} \leq \sqrt{D}$ for all $j$. 
    The second inequality follows from the fact $\norm{F}{\tsr C} = \norm{F}{\tsr C^{[p,n]}} = \norm{F}{\tsr{C}^{(n)}}$ for all $1 \leq p \leq n$.
    
    Tightness of \eqref{eq.mps_worst} follows from theorem \ref{thm.worst_tn}. 
    However, a canonical node $\tsr{C}^{(j)}$ cannot be rank-deficient as a matrix (i,n the orientation of canonicalization), thus we need some more work for tightness of \eqref{eq.cmps_worst}.
    
    Provided that $D_j = D$ for all $j$, we aim to show the inequality between first and last term of \eqref{eq.cmps_worst} is tight. 
    Let $\tn C = (G,\tsr{C}^{(1)},\, \tsr{C}^{(2)},\, \cdots \,  \tsr{C}^{(n)})$ be a canonical form centered at the last node. 
    We construct an instance such that equality is achieved. 
    We will write $\mat{M}_j$ as the matricized tensor $\tsr{C}^{(j)}_\rightarrow$ for $2 \leq j \leq n-1$, with left bond and physical dimension being columns, and right bond being rows. 
    For $j = 1$, we set $\mat{M}_1 = \tsr{C}^{(1)}$.
    Assign $\mat{M}_n = \vcr{u}_n \vcr{v}_n^\dagger + \mat{\Delta}$, where $\vcr{u}_n,~\vcr{v}_n$ are some normalized vectors and $\mat{\Delta}$ is a full rank matrix with $\norm{F}{\mat{\Delta}} \ll 1$. 
    Thus $\mat{M}_n$ is close to rank 1.
    Further, for all $2 \leq k \leq n-1$, choose isometry $\mat{M}_k$ so that $\mat{M}_k \vcr{u}_{k+1} = \vcr{u}_{k} \otimes \vcr{v}_{k}$ for some unit vectors $\vcr{u}_k$ and $\vcr{v}_k$, whose dimensions match the left bond and physical dimension of $C_k$, respectively. 
    Denote $\vcr{v}_1 := \mat{M}_1 \vcr{u}_2$, then we have
    \begin{equation}
        \tsr{C} = \tsr{C}^{[1, n-1]} \cdot (\vcr{u}_n \otimes \vcr{v}_n + \mat{\Delta}) =\bigotimes_{j = 1}^{n} \vcr{v}_j + \tsr{C}^{[1, n-1]} \mat{\Delta}. 
    \end{equation}
    Now let $\tsr{\delta}^{(i)} = \epsilon \sqrt{D} \cdot \vcr{u}_{i-1} \otimes \vcr{v}_i \otimes \vcr{u}_{i+1}^\dagger$ for $2 \leq i \leq n-1$, and at the ends set $\tsr{\delta}^{(1)} = \epsilon \sqrt{D} \vcr{v}_1 \otimes \vcr{u}_2^\dagger$ and $\tsr{\delta}^{(n)} = \epsilon \tsr{C}^{(n)}$.
    Clearly, this perturbation satisfies the relative relations. 
    Let the perturbed MPS be $\hat{\tsr{C}}$, then since vectors $\vcr{u}_j \AND \vcr{v}_j$ are normalized, up to an error of $\O(\epsilon^2)$,
    \begin{align*}
           &\norm{F}{\tsr{C} - \hat{\tsr{C}}} = \norm{F}{\tsr{C}} \sE_r(\tn{C}, \delta)\\
        =\ & 
            \norm{F}{\sum_{j = 1}^n \tsr{C}^{[1,\, j - 1]}\tsr{\delta}^{(j)} \tsr{C}^{[j+1,\, n]}} \\
        =\ & 
            \Bnorm{F}{ 
            \left(1 + (n-1)\sqrt{D}\right)\epsilon \bigotimes_{j = 1}^n \vcr{v}_j + \sum_{j = 1}^{n-1} \tsr{C}^{[1, j-1]} \tsr{\delta}^{(j)} \tsr{C}^{[j+1, n-1]} \mat{\Delta}+ \epsilon \mat{\Delta}
            }\\
        \geq\ &
            \Bnorm{F}{
            \left(1 + (n-1)\sqrt{D}\right)\epsilon \bigotimes_{j = 1}^n \vcr{v}_j
            } 
            - \left( \epsilon + \sum_{j=1}^{n-1}\norm{F}{\tsr{\delta}^{(j)}} \right) \norm{F}{\mat{\Delta}}\\
        \geq\ & 
            \epsilon\left(1 + (n - 1) \sqrt{D}\right) - \O\left(n \sqrt{D} \epsilon\right) \norm{F}{\mat{\Delta}}. 
    \end{align*}
    
    Since $\norm{F}{\mat{\Delta}}$ can be made arbitrarily small, the second term on the right hand side can be controlled by $\O(\epsilon^2)$. 
    Notice that in this construction, $\norm{F}{\tsr{C}} = \norm{F}{\vcr{u}_n \otimes \vcr{v}_n + \mat{\Delta}} = 1 + \O(\epsilon^2)$ for sufficiently small $\mat{\Delta}$. Therefore, we have
    \[
        \sE_r(\tn{C}, \delta) = \epsilon \left(1 + (n-1) \sqrt{D} \right) + \O(\epsilon^2).
    \]
    Hence, the bound \eqref{eq.cmps_worst} is tight up to a second order error.
\end{proof}

\begin{remark} \label{rmk.justify_center}
It should be clearly seen from the proof that a similar result holds if the canonical center is not the last node. 
The assumption is made here so that the statement and expressions are cleaner. 
If the center node is not the last one, one needs to adjust the orientation of 2-norms in \eqref{eq.mps_worst} and \eqref{eq.cmps_worst}.
\end{remark}

In corollary \ref{cor.cmps_worst}, one could loosen the second term of \eqref{eq.cmps_worst} by using a global upper bound $D$ for bond dimensions to simplify the expression. 
For a sufficiently large number of sites, if we truncate tensors and control the bond dimension $D_j \leq D$, then most of the sites will have bond dimension $D$ and the behavior at two tails, where $D_j < D$, can be neglected. 
Thus in the following, we compare these two bounds only using the uniform upper bound $D$.

One interpretation of these bounds is that when on average 
$\frac{
\bnorm{2}{\mat{T}^{[1, j-1]}_\rightarrow} \cdot \bnorm{F}{\tt{j}} \cdot \bnorm{2}{\mat{T}^{[j+1, n]}_\leftarrow}
}{
\norm{F}{\tsr{T}}
} > \sqrt{D}$, then  converting to canonical form reduces the error in the worst-case. 
If we compare the error bounds for an arbitrary MPS, we have the following.

\begin{corollary}[Comparison of MPS error bounds] \label{cor.comp_mps_err}
    Let $\tn{T} = \sT(\tt{1}, ~ \tt{2}, ~ \cdots ~ , \tt{n})$ be an arbitrary MPS. Let a canonical form of $\tn{T}$ be $\tn{C} = (G,\tsr{C}^{(1)},\, \tsr{C}^{(2)}, \cdots, \tsr{C}^{(n)})$. 
    Let $\hat{\tn{T}} \AND \hat{\tn{C}}$ be the relatively perturbed MPS of $\tn{T}$ and $\tn{C}$ by $\epsilon$-perturbations $(\alpha_i)_{i = 1}^n \AND (\beta_i)_{i = 1}^n$. 
    Suppose bond dimensions are uniformly bounded by $D$. 
    Then up to an error of order $\O(\epsilon^2)$,
    \begin{equation} \label{eq.comp_mps_err}
        \sup_\beta \sE_r(\tn{C}, \beta) \leq \frac{1 + (n - 1) \sqrt{D} }{ n } \sup_\alpha \sE_r(\tn{T}, \alpha).
    \end{equation}
    Moreover, provided that the bond dimension is $D$ across the MPS, the inequality is sharp, in the sense that the factor expressed in $n \AND D$ cannot be improved.
\end{corollary}

\begin{proof}
The inequality is not hard to establish,
since $\sup_\alpha \sE_r(\tn{T}, \alpha) \geq n\epsilon$ by choosing $\alpha_i = \epsilon \tt{i}$ for all $i \leq n$. 
To prove the tightness, we already know that the bounds \eqref{eq.mps_worst} and \eqref{eq.cmps_worst} are achieved when $\tn{T}$ and $\tn{C}$ are close to a product state network.
Construct $\tn{T} = \sT(\tt{1},\, \tt{2},\, \ldots \, , \tt{n})$ with $\tt k = \tsr{\overline{T}}^{(k)} + \mat{\Delta}^{(k)}$, where $\sT(\tsr{\overline{T}}^{(1)},~\tsr{\overline{T}}^{(2)},~\ldots,~\tsr{\overline{T}}^{(n)})$ is a product state network and $\mat{\Delta}^{(k)}$ are tensors with full rank matricization and small Frobenius norm. 
Thus $\tn{T}$ is nearly a tensor network product state, while $\tn{C}$ is close to a rank-1 at the center site.
The error is maximized in this case for the canonical form, and we have $\sup_\alpha \sE(\tn T, \alpha) = n\epsilon$ and $\sup_\beta \sE(\tn{C}, \beta) = \left( 1 + (n-1) \sqrt{D} \right) \epsilon$, up to an error of $\O(\epsilon^2)$. Thus \eqref{eq.comp_mps_err} is achieved.
\end{proof}

It is conjectured that a canonical MPS should have a better conditioning compared to general ones. 
However, corollary \ref{cor.comp_mps_err} is not as promising as expected, with a factor $\sqrt{D}$ in the numerator. 
Nevertheless, one can still consider canonical form provides an improvement in conditioning and error bound, in view that we essentially cannot bound in the other direction, i.e. given the canonical form $\tn{C}$, $\sup_\alpha \sE(\tn{T}, \alpha)$ is unbounded over all possible choices of $\tn{T}$.

\subsection{Stability and error analysis of PEPS} \label{ss.peps} 
Now we turn to applications to PEPS, and we will focus on stability of columnwise canonical PEPS, proposed in \cite{haghshenas2019conversion}. 
Despite the possibility to analyze sitewise perturbations, we will consider sitewise perturbation in a prescribed "center" column, and only columnwise perturbation in other columns. 
There are two main reasons for this. First, this is a more realistic model in practice. 
If the bound dimension is controlled (i.e. some truncation is made when representing a state by a PEPS) in algorithms like DMRG, then it is in general impossible to move from column to column exactly. 
One has to truncate the entire column to keep the bond dimension low, which leads to a dominating columnwise relative error. 
Second, despite that we can bound the entire error caused by sitewise perturbation, the worst-case bound is too large to be practically useful (as we will see, for a PEPS of size $m\times n$ and bond dimension $D$, even with columnwise perturbations, the error bound will be of order $nD^{m/2}$).

It is not hard to generalize the result for MPS to columnwise perturbation in PEPS. 
Recall that, as mentioned in introduction, we assume the PEPS is canonicalized towards the node at the lower right corner. 
We consider the model in which columnwise perturbation $(\mat{\Delta}^{(i)})_{i = 1}^{n-1}$ is introduced to a $m \times n$ PEPS $\tn{T}$, and sitewise perturbations $(\tsr{\delta}^{(i)})_{i = 1}^m$ are introduced to the center column $\tn{T}^{(\cdot, n)}$. 
The perturbed PEPS $\hat{\tn{T}}$ is given by
\begin{equation}
    \begin{cases}
    \hat{\tsr{T}}^{(\cdot, i)} = \tsr{T}^{(\cdot, i)} + \mat{\Delta}^{(i)} & \text{ if } 1 \leq i \leq n - 1, \\
    
    \hat{\tsr{T}}^{(j,n)} = \tsr{T}^{(j,n)} + \tsr{\delta}^{(j)} & \text{ for } 1\leq j \leq m.
    \end{cases}
\end{equation}
We will say $\hat{\tn{T}}$ is the perturbed PEPS of $\tn{T}$ by $(\Delta, \delta)$, and we denote the absolute error as
\begin{equation}
    \sE_a(\tn{T}, \Delta, \delta) :
    =  \norm{F}{\hat{\tsr{T}} - \tsr{T}},
\end{equation}
and the relative one as
\begin{equation}
    \sE_r(\tn{T}, \Delta, \delta) :
    = \frac{ \norm{F}{\hat{\tsr{T}} - \tsr{T}} }{ \norm{F}{\tsr{T}} }.
\end{equation}
We generalize the definition of $\epsilon$-perturbation. We say $(\Delta,\delta)$ is an $(\epsilon_1, \epsilon_2)$-perturbation to PEPS $\tn{T}$ if 
$\norm{F}{\mat{\Delta}^{(i)}} \leq \epsilon_1 \norm{F}{\tsr{T}^{(\cdot, i)}}$, and $\norm{F}{\tsr{\delta}^{(j)}} \leq \epsilon_2 \norm{F}{\tsr T^{(j,n)}} \FORAL i \leq n-1,~j\leq m$.

As in the previous part, we state the worst-case errors for general and canonical PEPS first, prove them together, and then compare them. 

\begin{corollary}[Worst-case error for general PEPS] \label{cor.peps_worst}
Let $\tn{T}$ be an arbitrary $m\times n$ PEPS. Suppose $(\Delta, \delta)$ is an $(\epsilon_1,\epsilon_2)$-perturbation to $\tn{T}$.
Then up to an error of $\O(\epsilon_1^2 + \epsilon_2^2)$,
\begin{equation} \label{eq.peps_worst}
    \begin{aligned} 
            \sE_r(\tn{T}, \Delta, \delta) \leq \ 
            & \epsilon_1 \cdot \sum_{i = 1}^{n-1} 
            \frac{ 
            \bnorm{2}{\mat{T}^{(\cdot, [1,i-1])}_\rightarrow} \, \norm{F}{\tsr{T}^{(\cdot, i)}} \, \bnorm{2}{\mat{T}^{(\cdot, [i+1,n])}_\leftarrow}
            }{
            \norm{F}{\tsr{T}}
            } \\ 
            &+ \sE_r(\tn{T}^{(\cdot, n)}, \delta) \cdot 
            \frac{
            \bnorm{2}{\mat{T}^{(\cdot,[1,n-1])}_\rightarrow} \norm{F}{\tsr{T}^{(\cdot, n)}}
            }
            {
            \norm{F}{\tsr{T}}
            },
    \end{aligned}
\end{equation}
where upper bound of $\sE_r(\tn{T}^{(\cdot, n)}, \delta)$ 
with $\epsilon = \epsilon_2$. 
Moreover, the bound is tight replacing $\sE_r(\tn{T}^{(\cdot, n)}, \delta)$ with the right hand side of \eqref{eq.mps_worst}.
\end{corollary}

\begin{corollary}[Worst-case error for canonical PEPS] \label{cor.cpeps_worst}
Let $\tn C = (\tsr C^{(i,j)})$ be a canonical $m\times n$ PEPS centered at the lower right corner. Suppose $(\Delta, \tn{\delta})$ is an $(\epsilon_1,\epsilon_2)$-perturbation to $\tn C$. 
Let $D_{ij}$ be the physical dimension at node $(i,~j)$, bounded by $D$. 
Then up to an error of $\O(\epsilon_1^2 + \epsilon_2^2)$,
\begin{equation} \label{eq.cpeps_worst}
    \begin{aligned}
        \sE_r(\tn C, \Delta, \tn{\delta}) \leq \ & \epsilon_1 \cdot \sum_{j = 1}^{n-1} \frac{ \left(\prod_{i = 1}^m D_{ij}^{1/2} \right)
        \norm{2}{\mat{C}^{(\cdot, [j+1,n])}_\leftarrow}
        }{
        \norm{F}{\tsr C}
        } + \sE_r(\tn C^{n,\cdot}, \delta) 
        \\
        \leq \ & 
        \epsilon_1(1 + (n-1)D^{m/2}) + \sE_r(\tn C^{(n,\cdot)}, \delta),
    \end{aligned}
\end{equation}
where $\sE_r(\tn C^{(n,\cdot)}, \delta)$ is bounded by \eqref{eq.cmps_worst} with $\epsilon = \epsilon_2$.
Moreover, provided that $D_{ij} = D$ for all $i,~j$, the bounds are tight replacing $\sE_r(\tn C^{(n,\cdot)}, \delta)$ by right hand side of \eqref{eq.cmps_worst}.
\end{corollary}

\begin{proof}
The proofs are applications of corollary \ref{cor.mps_worst} and \ref{cor.cmps_worst}, so we only present a sketch. 
For general PEPS, we first apply corollary \ref{cor.mps_worst} to the MPS formed by contracting columns of PEPS into tensors, i.e. the MPS $(G, \tsr{B}^{(1)}, ~ \tsr{B}^{(2)}, ~ \cdots, ~ \tsr{B}^{(n)})$, where $\tsr{B}^{(j)} = \sT(\tsr T^{(1,j)},~\tsr T^{(2,j)},~\cdots,~\tsr T^{(m,j)})$, and we bound the error in the last column $\tsr{B}^{(n)} = \sT(\tsr T^{(1,n)}, ~ \tsr T^{(2,n)}, ~ \cdots, ~ \tsr T^{(m,n)})$ by applying corollary \ref{cor.mps_worst} again. 
For tightness, we only need to prove that the equality-achieving cases for $\tsr{B}^{(n)} = \tsr{T}^{(\cdot,n)}$ as an MPS and $\B = \tn{T}$ are compatible. 
Indeed, to achieve the supremum of $\sE_r(\tn{T}^{(\cdot, n)}, \delta)$, we need $\tn{T}^{(\cdot, n)}$ to be close to a product state, which is compatible to another equality-achieving requirement in \eqref{eq.peps_worst} that $\tn{T}$ is close to a product state. 
The same applies to corollary \ref{cor.cpeps_worst}, with the aid of proof of corollary \ref{cor.cmps_worst}. We omit the details. 
\end{proof}

\begin{remark}
As mentioned in remark \ref{rmk.justify_center}, it is not hard to see from the proof that bound \eqref{eq.cmps_worst} holds when canonical center is not at the lower-right corner (the second inequality always holds. The first inequality can be adjusted when centered to other locations). 
\end{remark}

As in the previous MPS section, we compare the bounds we have. The result is similar to what we have for MPS.

\begin{corollary} [Comparison of PEPS error bounds] \label{cor.comp_peps_err}
    Let $\tn{T}$ be an arbitrary PEPS of size $m \times n$ and $\tn C$ be a canonical form of $\tn{T}$. Suppose $(\mat{\Delta}^{(i)},~\tsr{\delta}^{(j)})$ and $(\Xi_i,~\xi_j)$ are $(\epsilon_1,\epsilon_2)$-perturbations to $\tn{T}$ and $\C$ respectively.
    Then up to an error of order $\O(\epsilon_1^2 + \epsilon_2^2)$,
    \begin{equation} \label{comp_pepsf}
        \sup_{\Xi,~\xi} \sE_r(\C, \Xi, \xi) 
        \leq 
        \frac{
            \epsilon_1 \big(1 + (n - 1) D^{\frac{m}{2}} \big) + \epsilon_2 \big(1 + (m - 1) \sqrt{D}\big)
        }{ 
            \epsilon_1 (n - 1) + \epsilon_2 m 
        } 
        \sup_{\Delta, ~\delta} \sE_r(\tn{T}, \Delta, \delta).
    \end{equation}
    Moreover, the inequality is sharp, in the sense that the factor expressed in $n,~m, \AND D$ cannot be improved.
\end{corollary}

\begin{proof}
We follow the proof of corollary \ref{cor.comp_mps_err}. The inequality holds since $\sup_{(\Delta,\delta)} \sE_r(\tn{T}, \Delta, \delta) \geq \epsilon_1 (n - 1) + \epsilon_2 m$ by choosing
$\mat{\Delta}^{(j)} = \epsilon_1 \tsr{T}^{(\cdot,j)}$ and $\tsr{\delta}^{(i)} = \epsilon_2 \tsr{T}^{(i,n)}$ for all $i \leq m,~j \leq n-1$. 
To achieve equality with $\tn{T}$ whose non-center columns and nodes in center columns have full rank matricization, we consider perturbed product state PEPS, and send perturbation to 0. 
Specifically, write $\tsr{T} = \sT(\tsr{B}^{(1)},\, \tsr{B}^{(2)},\, \cdots \, , \tsr{B}^{(n)})$, where $\tsr{B}^{(j)} = \sT(\tsr T^{(1,j)},~\tsr{T}^{(2,j)},~\cdots,~\tsr{T}^{(m,j)})$. 
Now choose $\tsr{B}^{(j)} = \tsr{\overline{B}}^{(j)} + \tsr{A}^{(j)}$, where $\sT(\tsr{\overline{B}}^{(1)},~\tsr{\overline{B}}^{(2)},~\ldots,~\tsr{\overline{B}}^{(n)})$ is a product state network, 
and $\tsr{A}^{(j)}$ has full rank matricization, as in the proof of corollary \ref{cor.comp_mps_err} for $1 \leq j \leq n-1$. 
We can do the same trick regarding $\tsr{B}^{(n)}$ as an MPS.
Choose $\tsr{T}^{(i,n)} = \overline{\tsr{T}}^{(i,n)} + \tsr{A}^{(i,n)}$, where $\overline{\tsr{T}}^{(i,n)}$ form a product state $\tn{\overline{B}}$ and $\tsr{A}^{(i,n)}$ has full rank matricization.
As the magnitude of noise $\tsr{A}^{(j)} \AND \tsr{A}^{(i,n)}$ goes to zero, $\tn{T}$ and hence $\tn C$ are arbitrarily close to a product state. 
Thus equalities in \eqref{eq.peps_worst} and \eqref{eq.cpeps_worst} hold (and also \eqref{eq.mps_worst} and \eqref{eq.cmps_worst} for $\sE_r(\tn{T}^{(\cdot,n)}, \delta)$ and $\sE_r(\C^{(\cdot, n)}, \xi)$. 
Hence the bound is tight (up to a second order error).
\end{proof}


\section{Numerical Experiments} \label{s.num}

In this section, we perform three numerical studies to compare the worst case relative error of a general MPS and an MPS in canonical form.
The error is always quantified as relative error in Frobenius norm.

\subsection{Worst Case Error of Perturbing Center Node} \label{ss.center_perturb}
In this example, we consider only perturbing the center of a canonical form. 
The purpose is to see in what manner the canonical form outperforms a general form and how it is related to number of nodes $N$ and maximum bond dimension $D$.
Suppose $\T = \sT(\tt{1},~\tt{2},~\cdots,~\tt{N})$ is an MPS, and its canonical form centered at the middle node $c = \lfloor N/2 \rfloor$ is $\C = \sT(\tsr C^{(1)},~\tsr C^{(2)},~\cdots,~\tsr C^{(N)})$. 
For a given magnitude of perturbation $\epsilon$, we introduce perturbation tensors $\tsr{\delta}$ and $\tsr{\beta}$ such that $\norm{F}{\tsr{\delta}} = \epsilon \norm{F}{\tt{c}}$, $\norm{F}{\tsr{\beta}} = \epsilon \norm{F}{\tsr{C}^{(c)}}$. 
The perturbed MPS are $\hat{\T} = (G,~\tt{1},~\cdots,~\tt{c} + \tsr{\delta},~\cdots,~\tt{N})$ and $\hat{\tsr C} = \sT(\tsr C^{(1)},~\cdots,~\tsr C^{(c)} + \tsr{\beta},~\cdots,~\tsr C^{(N)})$. 

From Section \ref{s.wrst}, in particular corollary \ref{cor.one_site}, we know the relative error for general MPS $\T$ is $e_g := \epsilon \cdot \frac{\norm{2}{\mat{M}_{\tt{c}}} \norm{F}{\tt{c}}}{ \norm{F}{\tsr T}}$, and for canonical MPS $\C$ is $e_c := \epsilon$. 
We study the ratio of the errors $R := e_g / e_c$ as $N$ and $D$ varies. 
For each choice of $N$ and $D$, we sample 300 MPS, with each entry chosen uniformly from $[-1, 1]$. 
For each sampled MPS $\T$, we normalize it (i.e. make $\norm{F}{\tsr T} = 1$) and compute the canonical form $\C$ centered at node $c$, and then we explicitly compute $R$.
Note that for the ratio, value of $\epsilon$ does not matter. The results are shown in Figure \ref{fig.worst_center}. 

As expected, the ratio $R > 1$. From Figure \ref{fig.worst_center}, we see an sub-logarithmic increase of $R$ with respect to $D$ for this model, for all $N$ values tested. 
It suggests that given $D$, $R$ is about the same for different $N$ value. 
There is a decreasing trend in $R$ as $N$ decreases when $D$ is large (see $D = 64, ~ 128)$, but not significant in view of confidence intervals and the high absolute value -- about 4.3. 
This result substantiates the remark in Section \ref{ss.dmrg} that canonical forms improves the attainable accuracy in worst-case sense. 
Although the improvement turns out to be insensitive to $N$ in this model, canonicalization becomes more and more critical when $D$ gets large. 
\begin{figure}[H]
    \centering
    \includegraphics[width=\linewidth / 2 , height = 8cm]{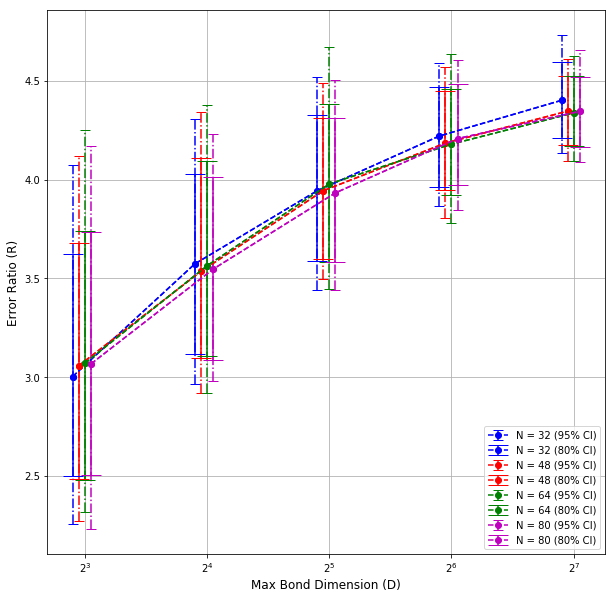}
    \caption*{\smaller{
    Fig.\ref{fig.worst_center} (Worst case error ratio of general form to canonical form). 
    Curves of $R$ with respect to bond dimension $D$ for different number of nodes $N$ are plotted (the $x$-axis is $\log$-scaled).
    Dots stand for the mean, and error bars stand for middle 80\% and 95 \% results} for 300 simulated cases. 
    (The horizontal difference between curves is artificially introduced for visibility. $\log_2D$ values are 3, 4, 5, 6, 7 for all $N$.
    }
    \captionlistentry{}
    \label{fig.worst_center}
\end{figure}
However, if we do NOT control bond dimension and allow it to grow exponentially from two ends, then the story is different. 
As $N$ grows, bond dimension at the center node grows, and by observation from last example, we expect $R$ grows. 
To further validate this, we do another test in which canonical center is chosen to be the second node, where bond dimension is independent of $N$. 
In both experiments, we sample 200 MPS for each value of $N$. 
As we can see in Figure \ref{fig.Dinf_worst}(a), a super-linear increase in $R$ is observed when canonical center is at the middle, whereas a decrease-flat curve is obtained when canonical center is at the second node. 

\begin{figure}[H]
    \centering
    \minipage{0.5\textwidth}
      \hspace{0.1\textwidth}
      \includegraphics[width=0.9\textwidth, height = 0.8\textwidth]{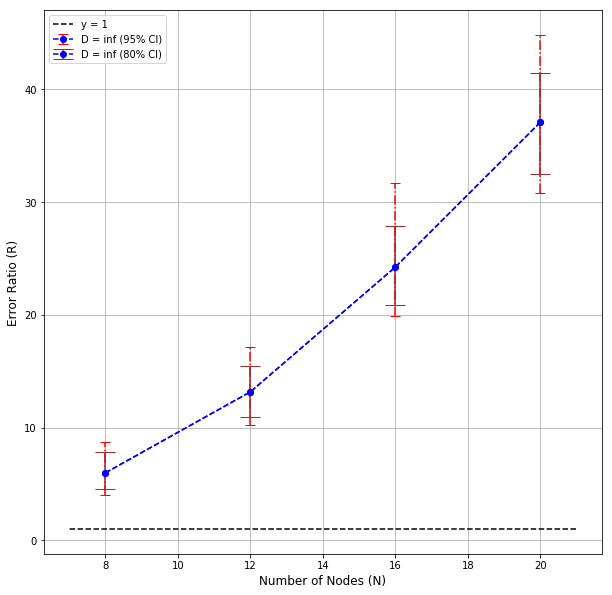}
    \endminipage\hfill
    \minipage{0.5\textwidth}%
      \includegraphics[width=0.9\textwidth, height = 0.8\textwidth]{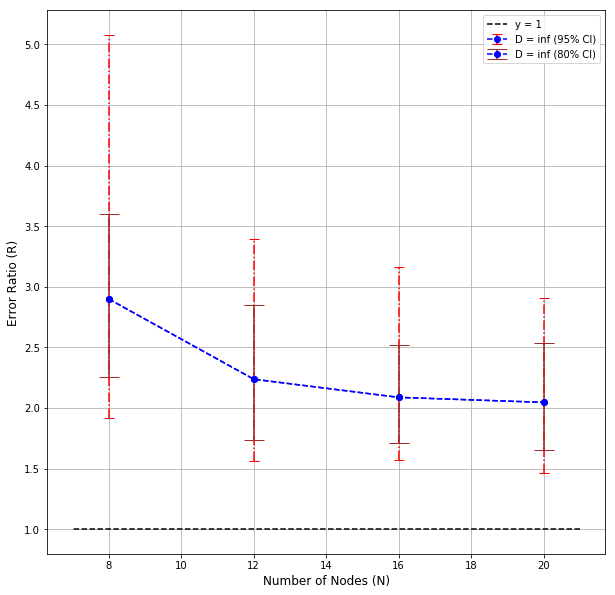}
      \hfill
    \endminipage
    \caption*{\smaller{
    Fig.\ref{fig.Dinf_worst} (Plot of $R$ with respect to $N$ when bond dimension is not controlled). \textbf{Left (a):} canonical center at the middle. \textbf{Right (b):} canonical center at the second node. Dots stand for the mean, and error bars stand for middle 80\% and 95 \% results for 200 generated MPS.}
    }
    \captionlistentry{}
    \label{fig.Dinf_worst}
\end{figure}

The decrease in Figure \ref{fig.Dinf_worst}(b) may be due to larger amount of cancellations in the right environment formed by $\tt{3},~ \cdots,~ \tt{N}$,
which leads to a smaller norm of the environment.
In conclusion, for this test model, we see that number of nodes does not seem to have remarkable effect on $R$, but the maximum bond dimension has a positive relation with $R$.

\subsection{All-site Perturbation}
In this example, we introduce perturbation to all sites of $\T$ and $\C$ defined in last example. 
Instead of directly compute the worst case errors established by corollary \ref{cor.mps_worst} and \ref{cor.cmps_worst}, which may not be uniformly tight for general MPS, we generate random perturbations to see how errors behave as $N$ and $D$ vary. 
Specifically, we generate, normalize, and canonicalize MPS in the same way as in Section \ref{ss.center_perturb}. For each $N$ and $D$, we sample 100 MPS. 
For each MPS $\T$ sampled, compute the canonical form $\C$, and we generate 200 perturbations and apply them to both $\T$ and $\C$ (scale each node of the perturbation so that the relative error of each node is $\epsilon$ respectively for $\T$ and $\C$). 
We compute error ratios (general form to canonical form) for these 200 perturbations, and set $R$ to be the largest one of 200 ratios. 
Statistically, this $R$ mimics the ratio of uniformly tight worst case errors of $\T$ and $\C$. 
In the test run, magnitude of perturbation is set as $\epsilon = 10^{-4}$. The results are plotted in Figure \ref{fig.Dfin_all}. 
\begin{figure}[H]
    \centering
    \includegraphics[width=\linewidth / 2, height = 8cm]{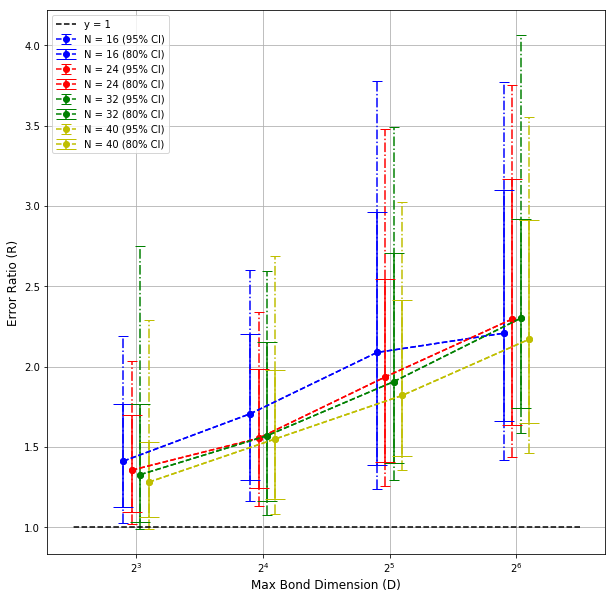}
    \caption*{\smaller{
    Fig.\ref{fig.Dfin_all} (All-site perturbation for general and canonical MPS). 
    Curves of $R$ with respect to bond dimension $D$ for different number of nodes $N$ are plotted. 
    Dots stand for the mean, and error bars stand for middle 80\% and 95 \% results for 100 generated MPS} (The horizontal difference between curves is artificially introduced for visibility. $\log_2D$ values are 3, 4, 5, 6 for all $N$.).
    }
    \captionlistentry{}
    \label{fig.Dfin_all}
\end{figure}

The resulting errors have higher variance than in single site perturbation (Figure \ref{fig.worst_center}) due to sampling.
First of all, in all cases the average ratio $R$ is remarkably above 1. 
This means even though there is a $\sqrt{D}$ factor for all-site perturbation, canonical form is still expected to have a better stability. 
This difference becomes more evident when $D$ gets larger. 
From Figure \ref{fig.Dfin_all}, we see $R$ increases as $D$ increases for all $N$. When $N = 16$, the curve starts to flatten when $D$ is large. This is expected since when $D$ is large enough and saturates bond dimension, in this case when $D \geq 256$, $R$ reaches its maximum and $D$ no longer has effect on $R$.
It is also noted that for fixed $D$, $R$ is not strongly related to $N$ as in single-site perturbation. 
One observation is that when $N$ is larger, the difference in error between canonical and general form tend to reveal with a larger bond dimension.
Overall, as bond dimension $D$ grows, converting to canonical form is more and more crucial for stability with respect to all-site perturbations.

Lastly, as in Section \ref{ss.center_perturb}, we study the influence of $N$ on $R$ when bond dimension is not controlled. This is also the maximum of $R$ with given $N$ as $D$ increases, as suggested by Figure \ref{fig.Dfin_all}.
As before, we generate 100 MPS for each value of $N$ and sample 200 perturbations for each generated MPS. 
We canonicalize MPS towards the middle node, apply scaled perturbation, and compare the error ratio $R$.
The magnitude of perturbation is $\epsilon = 10^{-4}$.
Due to exponential cost in time and memory and our relatively large sample size, we have to restrict the number of nodes.
The result is plotted in Figure \ref{fig.Dinf_all}.
\begin{figure}[H]
    \centering
    \includegraphics[width=\linewidth / 2 , height = 8cm]{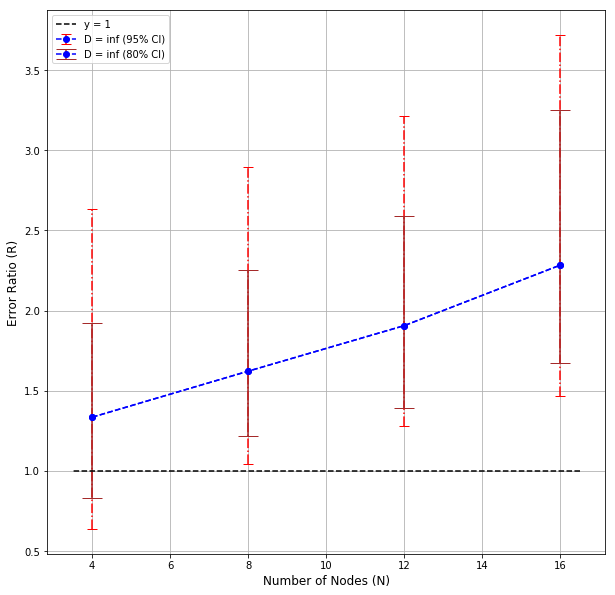}
    \caption*{\smaller{
    Fig.\ref{fig.Dinf_all}. Error ratio for all-site perturbation when bond dimension is not controlled. Dots stand for the mean, and error bars stand for middle 80\% and 95\% quantiles for 100 generated MPS.
    }}
    \captionlistentry{}
    \label{fig.Dinf_all}
\end{figure}

As expected, when we allow bond dimension to grow with $N$, we see an increase of $R$ with respect to $N$.
According to Figure \ref{fig.Dinf_all}, for $N = 16$, the maximum of $R$ over $D$ is about 2.0 $\sim$ 2.5, which is compatible with the result in \ref{fig.Dfin_all}.
If bond dimension is not controlled, the stability advantage of a canonical form is even more remarkable.

\subsection{Average-case Perturbation Error}
Lastly we illustrate our result on average-case error, theorem \ref{thm.ave}. We consider a small TN with 3 sites $\tsr A,~\tsr B,~\tsr C$ that form a triangle. Each node is a $D \times D$ matrix, and there is no out-going leg. Based our discussion in Section \ref{s.ave}, we aim to compare simulation results and theoretical results on average-case relative error under a sitewise perturbation with homogeneous variance $\sigma^2$ in each entry. For each $D$, we sample a TN with each entry being sampled from a uniform distribution on $[0, 1]$ (we did not use $[-1,1]$ just to avoid potential numerical instability when $\norm{F}{\tsr T}$ is close to 0). Then we sample 2000 random perturbations with each entry being sampled from a centered (mean 0) uniform distribution such that the standard deviation is scaled to be $\sigma = 10^{-3}$. Finally we compute the sample mean of the 2000 relative errors and compare to the theoretical value obtained from \eqref{eq.ave_err2}. The result is shown in Figure \ref{fig.ave}.

\begin{figure}[H]
    \centering
    \includegraphics[width=\linewidth / 2 , height = 8cm]{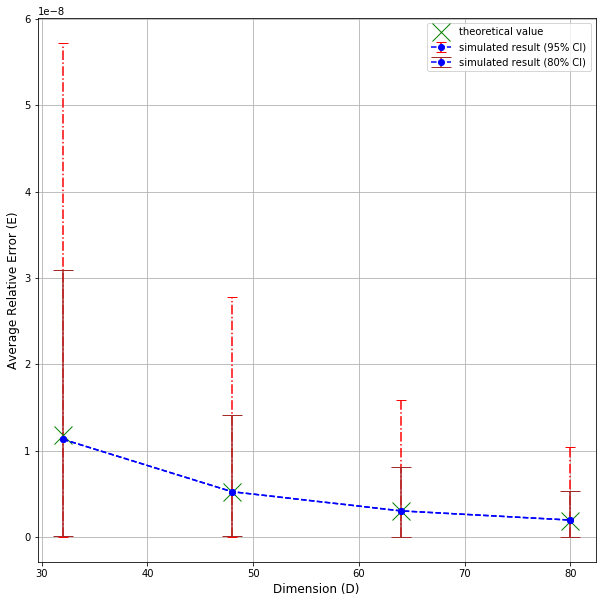}
    \caption*{\smaller{
    Fig.\ref{fig.ave}. Average-case relative error. Dots stand for the average of squared relative error of 2000 simulated results. Crossings are theoretical values proposed by theorem \ref{thm.ave}. Error bars stand for middle 80\% and 95\% quantiles of simulated results.
    }}
    \captionlistentry{}
    \label{fig.ave}
\end{figure}

Theoretical average and simulation average almost coincide. The theorem is confirmed. There is a decreasing trend in $E$ and its variance when $D$ gets larger. Though $\norm{F}{\mat{M}_{\tsr{T}}}$ increases with $D$, $\norm{F}{\tsr{T}}$ increases faster, and thus both decreasing trends are observed.

\vspace{2cm}


\section{Conclusion}
We have shown that numerical stability of a tensor network, including its condition number, worst-case and average-case perturbation errors, is characterized by its environment matrix. 
In particular, the worst-case perturbation error of a general tensor network is given as a solution to a quadratic system induced by the environment matrix, and explicit tight upper bounds are given. 
The average-case perturbation error for a general tensor network has also been derived and confirmed numerically.
We have put special focus on numerical advantage of a canonical form, and concluded theoretically that canonical forms provide benefits in reducing single-site perturbation error, and improving attainable accuracy for tensor network optimization algorithms. 
Numerical experiments have also confirmed that this benefit tends to be more and more significant for an MPS tensor network as the bond dimension grows. 

There are still challenges in canonicalization of tensor networks beyond MPS, for example, with PEPS the process is done in a costly iterative manner \cite{haghshenas2019conversion}.
Our work points to directions for improvement in canonicalization algorithms and design of such approaches for other tensor networks.
In particular, our error bounds suggest that it may suffice to maintain the tensor network in a ``well-conditioned'' tensor network gauge, in which environments are nearly orthogonal.

\bibliographystyle{plain}
\bibliography{paper}

\end{document}